\documentclass[12pt]{article}

\usepackage{amsthm,amsmath,amssymb,latexsym,color}

 \newcommand\E{{\mathbb E}}
\def\R{{\mathbb R}}
\def\P{{\mathbb P}}
\def\N{{\mathbb N}}
\def\Z{{\mathbb Z}}
\def\Net{{\mathcal N}}

\def\dist{{\rm d}}

\def\Exp{{\mathbb E}}
\def\M{{\mathcal M}}

\def\Im{{\rm Im}}
\def\diam{{\rm diam}}
\def\Vol{{\rm Vol}}

\def\Proj{{\rm P}}
\def\spn{{\rm span}}

\def\conv{{\rm conv}}
\def\B{{\rm BM}}
\def\S{\mathbb{S}}
\def\Id{{\bf I}}
\def\cov{{\rm cov}}
\def\rank{{\rm rank}}
\def\det{{\rm det}}
\newcommand\Event{{\mathcal E}}
\newcommand\Sphere{S}
\newcommand\BStat{{\mathcal B}}
\newcommand\BadBlocks{\mathcal I}

\title{When does a discrete-time random walk in $\R^n$ absorb the origin into its convex hull?}
\author{Konstantin Tikhomirov \ and \ Pierre Youssef\footnote{University of Alberta, Department of Mathematical and Statistical sciences. \hspace{3cm} \texttt{e-mail: \small ktikhomi@ualberta.ca ; pyoussef@ualberta.ca}}}

\begin{document}


\maketitle


\begin{abstract}
We connect this question to a problem of estimating
the probability that the image of certain random matrices
does not intersect with a subset of the unit sphere $\S^{n-1}$. In this way, the case
of a discretized Brownian motion is related to Gordon's
escape theorem dealing with standard Gaussian matrices. 
We show that for the random walk $\B_n(i), i\in \N$, the 
convex hull of the the first $C^n$ steps (for a sufficiently large
universal constant $C$) contains the origin with probability 
close to one. 
Moreover, the approach allows us to prove that
with high probability  the $\pi/2$-covering time of certain
random walks on $\S^{n-1}$ is of order $n$.
For certain spherical simplices on $\S^{n-1}$, we prove an extension of Gordon's
theorem dealing with a broad class of random matrices; as an application,
we show that $C^n$ steps are sufficient for the standard walk on $\Z^n$ to absorb
the origin into its convex hull with a high probability. 
Finally, we prove that the aforementioned bound is sharp in the following sense: for some universal constant $c>1$, the 
convex hull of the $n$-dimensional Brownian motion 
$\conv\{\B_n(t):\, t\in[1,c^{n}]\}$ does not contain the origin with probability close to one. 
\end{abstract}



\theoremstyle{plain}
\newtheorem{theor}{Theorem}
\newtheorem*{theoA}{Theorem~A}
\newtheorem*{theoB}{Theorem~B}
\newtheorem*{theoC}{Theorem~C}
\newtheorem*{theoD}{Theorem~D}
\theoremstyle{remark}
\newtheorem{rem}{Remark}
\theoremstyle{plain}
\newtheorem{prop}[theor]{Proposition}
\newtheorem{cor}[theor]{Corollary}
\newtheorem{lemma}[theor]{Lemma}
\newtheorem{defi}[theor]{Definition}

\section{Introduction}

The goal of this paper is to study certain convexity aspects of high-dimensional random walks. Given a discrete-time random walk $W(i)$ with values in $\R^n$, we are interested in estimating the number of steps
$N$ when the origin becomes an interior point of the convex hull of $\{W(i)\}_{i\le N}$.
This question was raised by I.~Benjamini and considered by R.~Eldan in \cite{MR3161524}.
Three models of random walks are treated in our paper:
a walk given by a discretization of the standard Brownian motion in $\R^n$,
the standard random walk on $\Z^n$ and a random walk on the unit sphere
$\S^{n-1}$. We employ a novel approach that
reduces the problem to a question about certain geometric properties of random matrices.
Random matrix theory has strong connections with asymptotic geometric analysis (see, for example, \cite{lama-book} and \cite{vershynin}); in particular, random matrices appear in
Gordon's escape theorem \cite{MR950977} and in various estimates of diameters of random
sections of convex sets \cite{MR0827766}, \cite{MR845980}. The interconnection between random walks,
random matrix theory and high-dimensional convex geometry is at the heart of our paper.

The {\it standard Brownian motion} with values
in $\R$ is defined as a centered Gaussian process $\B_1(t)$, $t\in [0,\infty)$, such that 
the covariance $\cov\left(\B_1(t)\right.$, $\left.\B_1(s)\right)=\min(t,s)$ for all $t,s\in[0,\infty)$. 
The Brownian motion in $\R^n$, denoted by $\B_n$, is a vector 
of $n$ independent one-dimensional Brownian motions.
We refer the reader to \cite{MR2604525} for extensive 
information on the process $\B_n$. 
Various properties of the convex hull of the Brownian motion in high dimensions were studied 
recently in \cite{MR3161524}, \cite{MR3210546} and \cite{volume-brownian};
in particular, results on interior and extremal points
of the convex hulls were obtained.
It is easy to see that the interior of $\conv\{\B_n(t):\,0<t<1\}$ (with ``$\conv$''
denoting the convex hull) contains the origin almost surely.
In the case when the domain $t\in (0,1)$ is replaced by a finite subset of the unit interval,
the origin is outside of the convex hull with a non-zero probability.
Our paper is motivated by the following problem which
in a more specific form was considered by Eldan in \cite{MR3161524}: \\

Let $t_1<t_2<\dots< t_N$ be points in $[0,1]$. How
is the probability that the origin 
belongs to the interior of $\conv\{\B_n(t_i):\, i\le N\}$ related to
the structure of the set $\{t_i\}_{i\le N}$?\\

In \cite{MR3161524}, the numbers $N$ and $t_1,t_2,\dots , t_N$ were generated by a homogeneous Poisson
point process in $[0,1]$. It was 
shown that when the expected number of generated points $N$ is greater than $e^{Cn\log(n)}$, the origin 
belongs to the interior of $\conv\{\B_n(t_i):i\le N\}$ with high probability
\cite[Theorem~3.1]{MR3161524}. A related result of \cite{MR3161524} dealing with the standard walk on 
$\Z^n$ states that, with probability close to one, $e^{Cn\log(n)}$ steps are sufficient
for the convex hull of the walk to absorb the origin.
It was not clear, however, whether the bound $e^{Cn\log(n)}$ was sharp. 
This question is addressed in the first main theorem of our paper:
\begin{theoA}\label{theorem-introduction}
There exists a constant $C>0$ such that
for any $n\in \N$ and $N\ge \exp(Cn)$ the following holds.
\begin{itemize}
\item Setting $t_i:=i/N$, $i=1,2,\dots ,N$, the set $\conv\{BM_n(t_i), i\le N\}$ contains the origin 
in its interior with probability at least $1-\exp(-n)$.
\item The convex hull of the first $N$ steps of the standard random walk on $\Z^n$ starting at $0$, contains 
the origin in its interior with probability at least $1-\exp(-n)$.
\end{itemize}
\end{theoA}
The first part of this theorem also holds when $\{t_i\}$ is a homogeneous Poisson process in $[0,1]$ of intensity
at least $\exp(Cn)$. Therefore, our result is strictly stronger than the bound proved in \cite{MR3161524}. 

Let us discuss optimality of the estimates in Theorem~A. 
Regarding  the second assertion, it was proved in 
\cite{MR3161524}
that if the number of steps $N$
is less than $\exp(cn/\log n)$ then with probability close to one
the origin does not belong to the interior 
of  the convex hull of the standard walk on $\Z^n$. 

For the first assertion of Theorem~A, we prove that it is optimal in the 
sense that the number of points $N$ must be
exponential in $n$ in order 
to have, say, $\P\{0\in\conv\{\B_n(t_i), i\le N\}\}\ge 1/2$. More precisely, 
we prove the following: 

\begin{theoB}\label{th-minimax}
There exist universal constants $c>0$ and $n_0\in\N$ with the following property: let $n\ge n_0$ and
$\B_n(t)$ ($0\le t<\infty$)
be the standard Brownian motion in $\R^n$. Then
$$\P\bigl\{0\in\conv\{\B_n(t):\,t\in [1,2^{cn}]\}\bigr\}\le \frac{1}{n}.$$
\end{theoB}

\begin{rem}
The bound $\frac{1}{n}$ in the above theorem can be replaced with $\frac{1}{n^L}$
for any constant $L>0$ at expense of decreasing $c$ and increasing $n_0$.
\end{rem}
 
The statement of Theorem~B is equivalent to the estimate 
\begin{equation}\label{eq-minmax-intro}
\P\bigl\{ \min_{u\in\Sphere^{n-1}}\max_{t\in[1,2^{cn}]} \langle u, \B_n(t)\rangle <0\bigr\} \geq 1-\frac{1}{n},
\end{equation}
where the quantity in the brackets is the ``minimax'' of $1$-dimensional Gaussian process
$\langle u,\B_n(t)\rangle$ indexed over $\Sphere^{n-1}\times [1,2^{cn}]$.
We note that a comparison theorem for the minimax of
doubly indexed Gaussian processes was obtained in \cite{MR800188}
(see also \cite[Corollary~3.13 and Theorem~3.16]{MR1102015}), and was the central ingredient in proving 
the escape theorem of \cite{MR950977}.
\bigskip

The second main result of our paper deals with discrete-time random walks on the sphere. 
For any $\theta\in (0,\pi/2)$, we consider a Markov chain $W_{\theta}$ with values in $\S^{n-1}$ such that 
the angle between two consecutive steps is $\theta$ 
(i.e. $\langle W_\theta (j), W_\theta(j+1)\rangle=\cos\theta$, $j\in\N$) and the direction from $W(j)$ to $W(j+1)$ is 
chosen {\it uniformly at random} in the sense that
for any $u\in \S^{n-1}$, the distribution of $W_\theta(j+1)$ conditioned on $W_\theta(j)=u$
is uniform on the $(n-2)$-sphere $\S^{n-1}\cap \{x\in\R^n:\,\langle x,u\rangle=\cos\theta\}$.

\begin{theoC}\label{th-intro-random-walk-sphere}
For any $\theta\in (0,\pi/2)$, there exist $L=L(\theta)$ and $n_0=n_0(\theta)$ depending 
only on $\theta$ such that the following holds: Let $n\ge n_0$ and $W_\theta$ be the process with values in $\S^{n-1}$
described above. Then for all $N\ge Ln$ we have 
$$
\P\bigl\{0\mbox{ belongs to }\conv\{W_{\theta}(i):\,i\le N\}\bigr\}\ge 1-\exp(- n).
$$ 
\end{theoC}

It follows from dimension considerations that the estimate of the number of steps is optimal up to a factor depending only on $\theta$. 
We note here that a related problem
for the standard spherical Brownian motion was studied in \cite{MR3161524}.

Let us outline the main ideas behind the proofs of Theorems~A and~C.
The following simple observation relates the question
about random walks to a problem dealing with random matrices:

Let $X(t)$ ($t\in[0,\infty)$ or $t\in\N\cup\{0\}$) be a random process
with values in $\R^n$, with $X(0)=0$; let $0=t_0<t_1< \dots <t_N$ be a collection of 
non-random points and assume that the increments $X(t_{i})-X(t_{i-1})$ are independent.
Define $A$ as the $N\times n$ random matrix with independent rows 
obtained by appropriately rescaling the increments $X(t_i)-X(t_{i-1})$, $i=1,2,\dots,N$. Then there exists a non-random 
$N\times N$ lower-triangular matrix $F$ such that the rows of $FA$ are precisely $X(t_i)$, $ i=1,2,\dots, N$. 
Thus, we can restate our problem about the convex hull of $X(t_i)$'s in terms 
of certain properties of the matrix $FA$. Namely, the convex hull of $X(t_i)$'s contains the origin in its interior 
if and only if for any unit vector $y$ in $\R^n$, the vector $FAy$ has at least one negative coordinate. 
Geometrically, this problem is reduced to estimating the probability that the image of $A$ {\it escapes}
(i.e.\ does not intersect) the set
$F^{-1}(\R_+^N)\cap \S^{N-1}$,
where $\R_+^N$ denotes the cone of positive vectors. For the standard Brownian motion, $A$ is the $N\times n$ standard 
Gaussian matrix. In this case, we apply Gordon's escape theorem \cite{MR950977} which estimates the probability that 
a random subspace uniformly distributed on the Grassmannian does not intersect with a given subset of $\S^{N-1}$. In
a more general case, 
when the image of $A$ is not uniformly distributed, Gordon's theorem cannot be applied. 
To treat that scenario, we prove a statement which can be seen as an extension
of Gordon's theorem to 
a broad class of random matrices, however, with considerable restrictions on the subsets of $\S^{N-1}$:

\begin{theoD}\label{intro-nongauss matrix th}
For any $\tau,\delta\in(0,1]$ and any $K>1$, there exist $L$ and $\eta>0$ depending only on $\tau$, $\delta$ 
and $K$ with the following property:
Let $N\ge Ln$ and let $A$ be an $N\times n$ random matrix with independent rows $(R_i)_{i\leq N}$ satisfying 
$$
\P\{ \langle R_i, y\rangle <-\tau\}\ge\delta, \text{ for any $y\in\S^{n-1}$ and any $i\leq N$.}
$$
Then for any $N\times N$ random matrix $F$, matrix $FA$ satisfies
\begin{align*}
\P\bigl\{\exists y\in \mathbb{S}^{n-1}, \; FAy\in\R_+^N\bigr\}&\le \exp(-\delta^2N/4)\\
&+\P\bigl\{\|A\|> K\sqrt{N}\bigr\}+\P\bigl\{ \Vert F-\Id\Vert >\eta \bigr\}.
\end{align*}
\end{theoD}
We use this result to deal with the random walk on $\Z^n$. 
For the random walks $W_\theta$ on
the sphere we follow, with some modifications, the same scheme as for processes in $\R^n$
with independent increments. 

\medskip
The paper is organized as follows.
Section~\ref{section-prel} contains preliminaries and notation.
Results about random matrices are given in Section~\ref{section-random-matrix-escape},
while corollaries for the Brownian motion and the standard random walk on $\Z^n$ are stated in Section~\ref{applications}. 
Section~\ref{section-sphere} is devoted to random walks on the sphere. Finally, we prove Theorem~B in Section~\ref{section-minimax}.

\section{Preliminaries}\label{section-prel}

In this section we introduce notation and discuss some classical or elementary facts.

For a finite set $I$, let $|I|$ be its cardinality.
Let $\R_+$ and $\R_-$ be the closed positive and negative semi-axes, respectively.
By $\{e_i\}_{i=1}^N$ we denote the standard unit basis in $\R^N$,
by $\|\cdot\|$ ~--- the canonical Euclidean norm and by $\langle \cdot,\cdot\rangle$ ~---
the corresponding inner product. Let $B_2^N$ and $\S^{N-1}$ be the Euclidean ball of radius $1$  in $\R^N$ and the unit sphere, respectively.

For $N\ge n$ and an $N\times n$ matrix $A$, let $s_{\max}(A)$ and $s_{\min}(A)$ be its largest and smallest singular values, respectively, i.e.\
$s_{\max}(A)=\|A\|$ (the operator norm of $A$) and $s_{\min}(A)=\inf\limits_{y\in \mathbb{S}^{n-1}}\|Ay\|$. When $A$ is an $N\times N$ invertible matrix, 
the condition number of $A$ is $\Vert A\Vert\cdot\Vert A^{-1}\Vert$.
Note that the condition number is equal to the ratio of the largest and
the smallest singular values of $A$.


Throughout this paper, $g$ denotes a standard Gaussian variable.
The following estimate is well known (see, for example, \cite[Lemma~VII.1.2]{Feller}):
\begin{equation}\label{feller ineq}
\P\{g\ge t\}=\frac{1}{\sqrt{2\pi}}\int\limits_t^\infty\exp(-r^2/2)\,dr
<\frac{1}{\sqrt{2\pi}t}\exp(-t^2/2),\;\;t>0.
\end{equation}
A random vector $X$ in $\R^n$ is {\it isotropic} if $\Exp X=0$ and the covariance matrix of $X$ is the identity i.e.\
$\Exp XX^t=\Id$.
The {\it standard Gaussian vector} $Y$ in $\R^n$ is a random vector with i.i.d.\ coordinates having the same law as $g$.
As a corollary of a concentration
inequality for Gaussian variables (see \cite[Theorem~4.7]{MR1036275} or \cite[Theorem~V.1]{MR0856576}), we have 
for any $\varepsilon >0$:
\begin{equation}\label{gaussian norm ineq}
\P\bigl\{(1-\varepsilon)\sqrt{n}\le\|Y\|\le (1+\varepsilon)\sqrt{n}\bigr\}\ge 1-2\exp(-\tilde c  \varepsilon^2n)
\end{equation}
for a universal constant $\tilde c>0$.
An  $N\times n$ matrix is called
the {\it standard Gaussian matrix}
if its  entries are 
i.i.d.\  having the same law as $g$. 
We denote this matrix by $G$ (and recall that
$N \ge n$). Then for any $t\ge 0$ we have
\begin{equation}\label{gaussian matrix ineq}
\begin{split}
\P&\bigl\{\sqrt{N}-\sqrt{n}-t\le s_{\min}(G)\le s_{\max}(G)\le \sqrt{N}+\sqrt{n}+t\bigr\}\\
&\ge 1-2\exp(-t^2/2)
\end{split}
\end{equation}
(see, for example, \cite[Corollary~5.35]{MR2963170}).

Given a vector $x\in\R^N$, we denote by $x_{+}$ and $x_{-}$ its positive and negative part,
respectively, i.e. 
$$x_{+}=\sum\limits_{i=1}^N{\max(0,\langle x,e_i\rangle)\,e_i} 
\quad\text{ and } \quad x_{-}=\sum\limits_{i=1}^N{\max(0,-\langle x,e_i\rangle)\,e_i}.$$
The following simple observation will be useful in the proof of the main theorems.

\begin{lemma}\label{negative part lemma}
Let $x,y\in\R^N$. Then $\Vert x_{-}\Vert\ge \Vert y_{-}\Vert -\Vert x-y\Vert$.
\end{lemma}
\begin{proof}
Writing $x=x_{+}-x_{-}$ and $y=y_{+}-y_{-}$, we obtain
\begin{align*}
\Vert x-y\Vert^2&=\Vert (x_{+}-y_{+})-(x_{-}-y_{-})\Vert^2\\
&=\Vert x_{-}-y_{-}\Vert^2+\Vert x_{+}-y_{+}\Vert^2-2\langle x_{+}-y_{+},x_{-}-y_{-}\rangle \\
&\ge \Vert x_{-}-y_{-}\Vert^2\\
&\ge \left(\Vert y_{-}\Vert-\Vert x_{-}\Vert\right)^2,
\end{align*}
where the first inequality in the above formula holds since $\langle x_{+}-y_{+},x_{-}-y_{-}\rangle$ is non-positive.
\end{proof}

Given a compact set $S\subset \R^N$, the {\it Gaussian width} of $S$ is defined by
$$
w(S):=\Exp\sup_{x\in S} \langle Y,x\rangle,
$$
where $Y$ is the standard Gaussian vector in $\R^N$ (see \cite{candes}, \cite{MR2989474}, \cite{vershynin}).  
The following is a consequence of Urysohn's inequality (see, for example, Corollary~1.4 in \cite{MR1036275}) 
and the relation between the Gaussian and the mean width:
\begin{equation}\label{urysohn}
\sqrt{N-1}\left(\frac{{\Vol_N}(S)}{{\Vol_N}(B_2^N)}\right)^{1/N}\le w(S).
\end{equation}

Given a convex cone $C$ in $\R^N$, the {\it polar cone} $C^*$ of $C$ is defined by
$$
C^*:= \{x\in\R^N,\ \langle x, y\rangle \le 0 \ \text{ for any } y\in C\}.
$$
The next Lemma provides a useful relation between the Gaussian widths of the parts of a convex cone and its polar
enclosed in the unit Euclidean ball.
The lemma is proved in \cite{MR2989474} for intersections of cones with the unit sphere (see
\cite[Lemma~3.7]{MR2989474});
we put it here in a version more convenient for us.

\begin{lemma}\label{width of a cone}
Let $C\subset \mathbb{R}^N$ be a nonempty closed convex cone. Then
$$
w\left(C\cap B_2^N\right)^2+w\left(C^*\cap B_2^N\right)^2\le N.
$$
\end{lemma}
\begin{proof}
For any $x\in \R^N$, let $P_{C}x:={\rm arg}\inf_{y\in C} \Vert x-y\Vert$ be the projection of $x$ onto $C$. 
It can be checked that each vector $x\in \R^N$ can be decomposed as
\begin{equation}\label{eq-decomposition-vector}
x=P_Cx+P_{C^*}x,
\end{equation}
with $\langle P_Cx, P_{C^*}x\rangle =0$. As before, let $Y$ be the standard Gaussian vector in $\R^N$.
Having decomposition \eqref{eq-decomposition-vector} in mind, we can write
$$
w(C\cap B_2^N)= \Exp\sup_{x\in C\cap B_2^N} \langle Y, x\rangle \le 
\Exp\sup_{x\in C\cap B_2^N} \langle P_{C}Y, x\rangle,
$$
where the last inequality holds since $\langle P_{C^*}Y, x\rangle\le 0$ for all $x\in C$. 
We deduce that  
\begin{equation}\label{eq1-cone}
w(C\cap B_2^N)\le \Exp \Vert P_CY\Vert.
\end{equation}
Now using the decomposition (\ref{eq-decomposition-vector}) and the above inequality, we obtain
\begin{equation}\label{eq2-cone}
w(C\cap B_2^N)^2\le\Exp \Vert P_CY\Vert^2= \Exp \Vert Y\Vert^2- \Exp \Vert P_{C^*}Y\Vert^2
= N- \Exp \Vert P_{C^*}Y\Vert^2.
\end{equation}
Note that \eqref{eq1-cone} applied to the cone $C^*$ yields $w(C^*\cap B_2^N)^2\le \Exp \Vert P_{C^*}Y\Vert^2$. 
Plugging it into \eqref{eq2-cone}, we complete the proof.
\end{proof}


\section{Escape theorems for random matrices}\label{section-random-matrix-escape}

In this section, we estimate the probability that the image of 
a random $N\times n$ matrix $A$ {\it escapes} the intersection of a given cone with the unit sphere $\S^{N-1}$
(we shall restrict ourselves to considering a special family of convex cones in $\R^N$).
Similar questions have attracted considerable attention recently in connection with the theory of 
compressed sensing \cite{candes}.

Given a closed subset $S\subset\S^{N-1}$,
the problem of estimating the probability
$\P\{\Im(A)\cap S=\emptyset\}$ can be treated in different ways.
One may look at it as the question of bounding the diameter of the random section
$\conv(S,-S)\cap \Im(A)$ of the convex set $\conv(S,-S)$: clearly, $\Im(A)\cap S=\emptyset$ if and only if
$\diam\bigl(\conv(S,-S)\cap \Im(A)\bigr)<2$.
The study of random sections of convex sets is a central theme in the area of asymptotic
geometric analysis and its importance has been highlighted in Milman's proof of Dvoretzky's theorem \cite{MR0856576}, \cite{MR1036275}.
The question of estimating diameters of random sections of proportional dimension
was originally considered in \cite{MR0827766} and \cite{MR845980} in the case when the corresponding random subspace
is uniformly distributed on the Grassmannian (i.e.\ the randomness is given
by a standard Gaussian matrix). More recently, results for much more general
distributions of sections given by kernels and images
of random matrices were obtained, among others, in papers \cite{MR2195339} and \cite{Mendelson}.
In our setting, however, these papers do not seem directly applicable as they provide estimates for diameters up to a
constant multiple: in particular, if a convex set $K$,
say, satisfies $K\subset B_2^N \subset 2K$, and $E$
is a random subspace given by a kernel or an image of a random matrix,
those results only give a trivial bound $\diam\bigl(K\cap E\bigr)<C$
for a large constant $C$.
At the same time,
if $S=\S^{N-1}\cap\R_+^N$ then it is easy
to show that $\conv(S,-S)\subset B_2^N\subset \sqrt{2}\,\conv(S,-S)$.

When the matrix $A$ is Gaussian, a way of estimating the probability $\P\{\Im(A)\cap S=\emptyset\}$ 
which is more suitable in our setting
is to apply the following result of Gordon (see Corollary~3.4 in \cite{MR950977}):

\begin{theor}[Gordon's escape theorem]   \label{escape}
  {\em
  Let $S$ be a subset of the unit Euclidean sphere $\mathbb{S}^{N-1}$ in $\mathbb{R}^N$.
  Let $E$ be a random $n$-dimensional subspace of $\R^N$,
  distributed uniformly on the Grassmannian with respect to the associated Haar measure.
  Assume that
  $
    w(S) < \sqrt{N-n}.
  $
  Then
  $
  E \cap S = \emptyset
  $
  with probability at least
  $$
  1 - 3.5 \; \exp \Big( -\frac{1}{18}\Bigl( \frac{N-n}{\sqrt{N-n+1}} - w(S) \Bigr)^2 \Big).
  $$
  }
\end{theor}

\bigskip

For the standard Gaussian matrix $G$, its image is uniformly distributed on the
Grassmannian, and Gordon's result provides an efficient estimate of
probability $\P\{\Im G\cap S=\emptyset\}$, as long as we have control over the Gaussian width of 
the set $S$. In our setting, the choice of $S$ is determined by the applications to
random walks; in fact, $S$ shall always be a spherical simplex satisfying certain additional assumptions. 
A standard approach would be to bound $w(S)$ in terms of the covering numbers of $S$ 
using the classical Dudley's inequality (see, for example, \cite[Theorem~11.17]{MR1102015}). 
However, in our case the set $S$ is relatively large,
so the upper bound given by Dudley's inequality is trivial (greater than $\sqrt{N}$). 
Instead, we will estimate the Gaussian width of $S$ using the following proposition which is a direct consequence
of Lemma~\ref{width of a cone} and inequality \eqref{urysohn}:

\begin{prop}\label{prop-width-cone}
Let $C$ be a convex cone in $\R^N$ and denote by $C^*$ its polar cone. Then 
$$
w(C\cap B_2^N)^2\leq N-(N-1)\left(\frac{{\Vol_N}(C^*\cap B_2^N)}{{\Vol_N}(B_2^N)}\right)^{2/N}.
$$
\end{prop}

The next theorem will be applied in Sections~\ref{applications} and~\ref{section-sphere} to
the discretized Brownian motion and to random walks on the sphere.

\begin{theor}\label{gaussian-escape}
For any $\gamma \in (0,1]$ there exist positive $L$, $\kappa$  and $\eta$ depending on $\gamma$ 
such that the following is true:
For $N\ge Ln$, let $F$ be an $N\times N$ random matrix and $\tilde F$ be a deterministic invertible $N\times N$ 
matrix with the condition number satisfying
$\Vert \tilde F\Vert\cdot \Vert\tilde F^{-1}\Vert \le \gamma^{-1}$. If 
$G$ is the $N\times n$ 
standard Gaussian matrix,
then
$$
\P\bigl\{\exists y\in \mathbb{S}^{n-1}, \; FGy\in \R^N_{+}\bigr\}
\le    5.5 \; \exp ( -\kappa N )+ \P\bigl\{ \bigl\Vert F-\tilde F\bigr\Vert > \eta\Vert \tilde F\Vert\bigr\}.
$$
The statement holds with $L=64/\gamma^2$, $\kappa= 2L^{-2}/9$ and   
$\eta=\gamma/4L$.\end{theor}

\begin{proof} 
Let $\gamma\in (0,1)$ and take $L,\kappa,$ and $\eta$ as stated above.
In view of Lemma~\ref{negative part lemma} we have
\begin{align}
\P&\left\{\exists y\in \mathbb{S}^{n-1}, \; \left(FGy\right)_{-}=0\right\}\nonumber\\
&\le\P\left\{\exists y\in \mathbb{S}^{n-1}, \;\|(\tilde FGy)_{-}\|\le\|(F-\tilde F)Gy\|\right\}\nonumber\\
&\le\P\left\{\exists y\in \mathbb{S}^{n-1}, \;\|(\tilde FGy)_{-}\|\le \eta\Vert \tilde F\Vert\cdot\|G\|\right\}\nonumber 
+ \P\{\Vert F-\tilde F\Vert > \eta\Vert \tilde F\Vert\}.
\end{align}
Further,
\begin{align}
\P&\left\{\exists y\in \mathbb{S}^{n-1}, \;\|(\tilde FGy)_{-}\|\le \eta\Vert \tilde F\Vert\cdot\|G\|\right\}\nonumber\\
&\le\P\left\{\exists y\in \mathbb{S}^{n-1}, \; \tilde FGy \in \R^N_{+}+\eta\Vert \tilde F\Vert\cdot\|G\| B_2^N\right\}\nonumber\\
&\le\P\Bigl\{\exists y\in \mathbb{S}^{n-1}, \; \frac{Gy}{\Vert Gy\Vert} \in {\tilde F}^{-1}(\R^N_{+})
+\eta\Vert \tilde F\Vert \frac{\|G\|}{s_{\min}(G)} {\tilde F}^{-1}
\left(B_2^N\right)\Bigr\}\nonumber\\
&\le\P\Bigl\{\exists y\in \mathbb{S}^{n-1}, \; \frac{Gy}{\Vert Gy\Vert}
\in {\tilde F}^{-1}(\R^N_{+})+2\eta \cdot\gamma^{-1}B_2^N\Bigr\}\nonumber\\
&\hspace{1cm}+ \P\bigl\{\Vert G\Vert>2s_{\min}(G)\bigr\}\nonumber\\
&\le\P\Bigl\{{\rm Im}(G)\cap \bigl( {\tilde F}^{-1}(\R^N_{+})+2\eta \cdot\gamma^{-1}B_2^N\bigr)\cap \mathbb{S}^{N-1}\neq \emptyset\Bigr\}\label{p of escape gauss}
+ 2e^{-N/128},
\end{align}
where the last estimate follows from \eqref{gaussian matrix ineq}.

To control the probability of escaping in \eqref{p of escape gauss} with help
of Theorem~\ref{escape}, we have to estimate the Gaussian width of the set
$$\Gamma:=\bigl( {\tilde F}^{-1}(\R^N_{+})+2\eta\cdot \gamma^{-1}B_2^N\bigr)\cap \mathbb{S}^{N-1}.$$ 
Note that $\Gamma\subset (1+2\eta\cdot\gamma^{-1}){\tilde F}^{-1}(\R^N_{+})\cap B_2^N + 2\eta\cdot \gamma^{-1}B_2^N$. 
Therefore 
\begin{equation}\label{width-gamma}
w(\Gamma)\le (1+2\eta\cdot\gamma^{-1})\cdot w\left({\tilde F}^{-1}(\R^N_{+})\cap B_2^N\right)+  2\eta\cdot \gamma^{-1}\sqrt{N}.
\end{equation}
It remains to bound the Gaussian width of ${\tilde F}^{-1}(\R^N_{+})\cap B_2^{N}$.
Denote by $C$ the cone ${\tilde F}^{-1}(\R_{+}^N)$ and note that $C^*={\tilde F}^t(\R_{-}^N)$. Then we have
\begin{align*}
\Vol_N\bigl({\tilde F}^t(\R_{-}^N)\cap B_2^N\bigr)
&=\vert \det({\tilde F})\vert\cdot \Vol_N\bigl(\R_{-}^N\cap ({\tilde F}^t)^{-1}(B_2^N)\bigr)\\
&\ge \vert \det\bigl({\tilde F}\bigr)\vert\cdot\Vert {\tilde F}\Vert^{-N}\cdot 
\Vol_N\bigl(\R_{-}^N\cap B_2^N\bigr).
\end{align*}
Since $\vert \det\bigl({\tilde F}\bigr)\vert \ge \Vert {\tilde F}^{-1}\Vert^{-N}$, we get
$
\Vol_N\left(C^*\cap B_2^N\right)\ge \left(\gamma/2\right)^N\cdot 
\Vol_N\left( B_2^N\right).
$
Now, applying Proposition~\ref{prop-width-cone},  we deduce that 
\begin{equation}\label{width-our-cone}
w(C\cap B_2^N)\le \sqrt{(1-\gamma^2/8)N}.
 \end{equation}

Putting \eqref{width-gamma} and \eqref{width-our-cone} together, we get that
$$w(\Gamma)\le  \left(1+4\eta\cdot \gamma^{-1} -\gamma^2/16\right)\sqrt{N}.$$
The proof is finished by a straightforward application of Theorem~\ref{escape}.
\end{proof}

As we will see in the next sections, Theorem~\ref{gaussian-escape} provides a way 
to deal with the standard Brownian motion in $\R^n$ and
random walks $W_\theta$ on the sphere. To treat the standard walk on $\Z^n$,
we shall derive a statement covering a rather broad class of random matrices.
Let us introduce the following


\begin{defi}
A random variable $\xi$ is said to have property $\mathcal{P}(\tau,\delta)$ (or safisfy condition
$\mathcal{P}(\tau,\delta)$) for some $\tau,\delta\in(0,1]$ if $$\P\{\xi<-\tau\}\ge\delta.$$
A random vector $X$ in $\R^n$ is said to have property $\mathcal{P}(\tau,\delta)$ for $\tau,\delta\in(0,1]$ 
if for any $y\in\mathbb{S}^{n-1}$, the random variable $\langle X,y\rangle$ satisfies $\mathcal{P}(\tau,\delta)$.
\end{defi}

Obviously, the above property holds (for some $\tau$ and $\delta$) for any non-zero r.v.\ $\xi$ with $\Exp\xi=0$.
As the next elementary lemma shows, with some additional assumptions on moments of $\xi$, the numbers
$\tau$ and $\delta$ can be chosen as certain functions of the moments:

\begin{lemma}\label{lemma about negative spread}
Any random variable $\xi$ such that
$\Exp\xi=0$, $\Exp\xi^2=1$ and $\Exp|\xi|^{2+\varepsilon}\le B<\infty$ for some $\varepsilon>0$,
has the property $\mathcal{P}(\tau,\delta)$,
with $\tau$ and $\delta$ depending only
on $\varepsilon$ and $B$.
\end{lemma}
\begin{proof}
Indeed, an easy calculation shows
that such $\xi$ satisfies
$$\int\limits_{{L_\xi}^2}^\infty\P\{\xi^2\ge t\}\,dt\le\frac{1}{2}$$
for some $L_\xi>0$ depending only on $B$ and $\varepsilon$.
Then
$$
\Exp|\xi|\ge\int\limits_{0}^{L_\xi}\P\{|\xi|\ge t\}\,dt
\ge \frac{1}{2L_\xi}\int\limits_0^{{L_\xi}^2}
\P\{\xi^2\ge t\}\,dt
\ge\frac{1}{4L_\xi},
$$
implying, as $\Exp\max(0,-\xi)=\frac{1}{2}\Exp|\xi|$,
\begin{align*}
\frac{1}{8L_\xi}
&\le\int\limits_0^\infty\P\{\xi\le -t\}\,dt\\
&\le\int\limits_0^{8L_\xi}\P\{\xi\le -t\}\,dt + \int\limits_{64{L_\xi}^2}^\infty\frac{1}{2\sqrt{t}}\P\{\xi^2\ge t\}\,dt\\
&\le\int\limits_0^{8L_\xi}\P\{\xi\le -t\}\,dt+\frac{1}{16L_\xi}.
\end{align*}
Hence, $\P\{\xi< -2^{-5}{L_\xi}^{-1}\}\ge 2^{-8}{L_\xi}^{-2}$.
\end{proof}

The following theorem will be used to treat the standard walk on $\Z^n$:

\begin{theor}\label{nongauss matrix th}
For any $\tau,\delta\in(0,1]$ and any $K>1$, there exist $L$ and $\eta>0$ depending only on $\tau$, $\delta$ 
and $K$ with the following property:
Let $N\ge Ln$ and let $A$ be an $N\times n$ random matrix with independent rows having
property $\mathcal{P}(\tau,\delta)$.
Then for any $N\times N$ random matrix $F$, matrix $FA$ satisfies
\begin{align*}
\P\bigl\{\exists y\in \mathbb{S}^{n-1}, \; FAy\in\R_+^N\bigr\}&\le \exp(-\delta^2N/4)\\
&+\P\bigl\{\|A\|> K\sqrt{N}\bigr\}+\P\bigl\{ \Vert F-\Id\Vert >\eta \bigr\}.
\end{align*}
\end{theor}

\begin{proof}
Define $L$ as the smallest positive number satisfying
$$\Bigl(\frac{3}{\eta}\Bigr)^{1/L}\le\exp(\delta^2 /4),$$
where $\eta:=\frac{\sqrt{\delta}\,\tau}{2\sqrt{2}K}$. 
Now, take any admissible $N\ge Ln$ and let $A$ and $F$ be as stated above.

Let $\Net$ be an $\eta$-net on $\mathbb{S}^{n-1}$ of cardinality at most $\bigl(\frac{3}{\eta}\bigr)^n$.
In view of Lemma~\ref{negative part lemma} we have
\begin{align}
\P&\bigl\{\exists y\in \mathbb{S}^{n-1}, \; FAy\in\R_+^N\bigr\}\nonumber\\
&\le\P\bigl\{\exists y\in \mathbb{S}^{n-1}, \;\|(Ay)_{-}\|\le\|(F-\Id)Ay\|\bigr\}\nonumber\\
&\le \P\bigl\{\exists y\in \mathbb{S}^{n-1}, \;\|(Ay)_{-}\|
\le \eta\|A\|\bigr\} +\P\bigl\{ \Vert F-\Id\Vert >\eta \bigr\}\nonumber\\
&\le\P\bigl\{\exists y'\in \Net, \;\|(Ay')_{-}\|
\le 2\eta\|A\|\bigr\}+\P\bigl\{ \Vert F-\Id\Vert >\eta \bigr\}.\nonumber
\end{align}
Further,
\begin{align}
\P&\bigl\{\exists y'\in \Net, \;\|(Ay')_{-}\|\le 2\eta\|A\|\bigr\}\nonumber\\
&\le \P\bigl\{\exists y'\in \Net, \;\|(Ay')_{-}\|\le 2K\eta\sqrt{N}\bigr\}
+\P\bigl\{\|A\|> K\sqrt{N}\bigr\}.\label{nongaussian prop. net}
\end{align}
Fix any $y'\in\Net$. For all $i=1,2,\dots,N$, the random variable $\langle Ay',e_i\rangle$ satisfies 
the property $\mathcal{P}(\tau, \delta)$. For any $i\le N$, denote by $\chi_i$ the indicator function of the event 
$\{ \langle Ay',e_i\rangle < -\tau\}$. Then $(\chi_i)_{i\le N}$ are independent and $\mathbb{E}\chi_i\ge \delta$. 
Applying Hoeffding's inequality (see \cite[Theorem~1]{MR0144363}), we get
\begin{align*}
\P\Bigl\{|\{i\le N:\,\langle Ay',e_i\rangle< -\tau\}|\le \frac{\delta N}{2}\Bigr\}
&\le\P\Bigl\{ \frac{1}{N} \sum_{i\le N} (\chi_i-\mathbb{E}\chi_i) \le -\frac{\delta}{2}\Bigr\}\\
&\le \exp(-\delta^2 N/2).
\end{align*}
Therefore for any fixed $y'\in \Net$, we have
\begin{align*}
\P\bigl\{\|(Ay')_{-}\|\le 2K\eta\sqrt{N}\bigr\}
&\le \P\bigl\{|\{i\le N:\,\langle Ay',e_i\rangle\le -\tau\}|\le 4K^2\eta^2N/\tau^2\bigr\}\\
&\le \exp(-\delta^2 N/2).
\end{align*}
Combining the last estimate with \eqref{nongaussian prop. net} and the upper estimate for $|\Net|$, we get
\begin{align*}
\P&\bigl\{\exists y\in \mathbb{S}^{n-1}, \; FAy\in\R_+^N\bigr\}\\
&\le \Bigl(\frac{3}{\eta}\Bigr)^n\exp(-\delta^2 N/2)+\P\bigl\{\|A\|> K\sqrt{N}\bigr\}+\P\bigl\{ \Vert F-\Id\Vert >\eta \bigr\}.
\end{align*}
The result follows by the choice of $L$.
\end{proof}

\begin{rem}
Theorem~\ref{nongauss matrix th}, applied to the Gaussian matrix $G$, 
gives a weaker form of Theorem~\ref{gaussian-escape} (with more restrictions on the choice of $F$).
Let us emphasize that the theorems do not require $F$ to be independent from $G$.
This will be important in Section~\ref{section-sphere}.
\end{rem}


\section{Applications to random walks in $\R^n$}\label{applications}

In this section, we will apply the statements about random matrices to the Brownian motion and 
the standard walk on $\Z^n$. 

\begin{cor}\label{geometric growth}
For any $K>1$, there are constants $L$ and $\kappa$ depending only on $K$ such that the following holds.
 Let $N\ge Ln$ and $t_1,\dots,t_N$ be such that $t_i\ge K\cdot t_{i-1}$ for any $i=2\dots  N$ and $t_1>0$. 
Then 
$$
\P\bigl\{0\mbox{ belongs to the interior of }\conv\{\B_n(t_i):\,i\le N\}\bigr\}
\ge   1 - 5.5  \exp ( -\kappa N ).
$$
\end{cor}

\begin{proof}
Let $c_K:=1+(K-1)^{-1/2}\sum_{j\ge 0} K^{-j/2}$ and $\gamma:= c_K^{-1}\cdot(1+(K-1)^{-1/2})^{-1}$ be two constants depending only on $K$ and take 
$L=64/\gamma^2$ and $\kappa:= 2L^{-2}/9$.

Denote $\delta_1:=\sqrt{t_1}$ and 
$\delta_i:=\sqrt{t_i-t_{i-1}}$ for any $i=2\dots N$. 
Observe that for any $j< i$, we have $\delta_i \ge K^{\frac{i-j-1}{2}}\sqrt{K-1}\cdot\delta_j$. 

Define $F$ as the $N\times N$ lower triangular matrix whose entries are given by $f_{ii}=1$ for any $i\le N$ and 
$f_{ij}=\frac{\delta_j}{\delta_i}$ for any $i>j$. One can easily check that $\Vert F\Vert\le c_K$. Moreover, the inverse 
of $F$ is  a lower bidiagonal matrix with $1$ on the main diagonal and $(\delta_i/\delta_{i+1})_{i<N}$ on the diagonal below. 
Hence $\Vert F^{-1}\Vert \le 1+(K-1)^{-1/2}$, and the condition number of $F$ satisfies 
$$
\Vert F\Vert\cdot\Vert F^{-1}\Vert \le \gamma^{-1}.
$$
Let $(R_i)_{i\le N}$ be the rows of $FG$. One can check that $R_i= \B_n(t_i)/\delta_i$ and therefore 
$$
0\in\conv\{\B_n(t_i):\,i\le N\}\Leftrightarrow 
0\in\conv\{R_i:\,i\le N\}
$$
Note that, by a standard separation argument, $0$ does not belong to the interior of $\conv\{R_i:\,i\le N\}$ 
if and only if $\rank (FG) <n$ or there is a vector $y\in \mathbb{S}^{n-1}$ such that 
$\langle FGy,e_i\rangle=\langle y,R_i\rangle \ge 0$ for any $i\le N$, where 
$(e_i)_{i\le N}$ denotes the canonical basis of $\R^N$.
Since with probability one we have $\rank(FG)=n$, the result follows by applying Theorem~\ref{gaussian-escape} with $\tilde F:=F$.
\end{proof}

Suppose $(t_i)$ is a finite increasing
sequence of points in $[0,1]$.
The above statement tells us that if $(t_i)$
contains a geometrically growing subsequence of
length $Ln$ for an appropriate $L>0$ then with high probability
the origin of $\R^n$ is contained in the
interior of $\B_n(t_i)$'s.
We shall apply this result to the case when
the $t_i$'s are generated by the Poisson point process 
independent from $\B_n$.

Recall that {\it the homogeneous Poisson point process in $[0,1]$} of intensity $s>0$ is a random discrete measure $N_s$
on $[0,1]$ such that
$1)$ for each Borel subset $B\subset[0,1]$, the random variable $N_s(B)$
has the Poisson distribution with parameter $s\mu(B)$, where $\mu$ is the usual Lebesgue measure on $\R$, and
$2)$ for any $j\in\N$ and pairwise disjoint Borel sets $B_1,B_2\dots,B_j\subset[0,1]$, the random variables
$N_s(B_1)$, $N_s(B_2),\dots,N_s(B_j)$ are jointly independent.
The measure $N_s$ admits a representation of the form
$$N_s=\sum\limits_{i=1}^\tau\delta_{\xi_i},$$
where $\xi_1,\xi_2,\dots$ are i.i.d.\ random variables uniformly distributed on $[0,1]$,
$\delta_{\xi_i}$ is the Dirac measure with the mass at $\xi_i$ and
$\tau$ is the random non-negative integer with the Poisson distribution with parameter $s$.

Theorem~3.1 of \cite{MR3161524} states that
if $\tau$ and the points $\xi_1,\xi_2,\dots,\xi_\tau$
are generated by the homogeneous PPP in $[0,1]$
of intensity $s\ge n^{Cn}$ then the convex hull
of $\B_n(\xi_i)$'s contains the origin in its interior with
probability at least $1-n^{-n}$.
In our next statement, we weaken the assumptions on $s$
at expense of decreasing the probability to $1-\exp(-n)$:
\begin{cor}
There is a universal constant $\tilde C>0$ with the following property:
Let $n\in\N$ and let $\B_n(t)$, $t\in[0,\infty)$, be the standard Brownian motion in $\R^n$.
Further, let $\tau$ and the points $\xi_1,\xi_2,\dots,\xi_\tau$
be given by the homogeneous Poisson process on $[0,1]$
of intensity $s\ge\exp(\tilde Cn)$, which is independent from $\B_n(t)$.
Then
$$\P\bigl\{0\mbox{ belongs to the interior of }\conv\{\B_n(\xi_i):\,i\le \tau\}\bigr\}\ge 1-\exp(-n).$$
\end{cor}
\begin{proof}
Let $K:=2$ and $\kappa,L$ be as in Corollary~\ref{geometric growth}.
Then we define the constant
$\tilde C:=\max\bigl(\frac{32}{\kappa},8L\bigr)$.
Let $n\in\N$ and let $N_s$ be as stated above.
Take $m:=\lfloor \tilde Cn\rfloor$ and
$$I_{1}:=[0,K^{-{m}+1}];\;\;I_j:=(K^{j-m-1},K^{j-m}],\;
j=2,3,\dots,m.$$
From the definition of $N_s$, we have
\begin{align*}
\P\bigl\{N_s(I_j)>0\mbox{ for all }j=1,2,\dots,m\bigr\}
&\ge 1-\sum\limits_{j=1}^m \exp\bigl(-s\mu(I_j)\bigr)\\
&\ge 1-m\exp\bigl(-sK^{-m}\bigr).
\end{align*}
In particular, with probability at least $1-m\exp\bigl(-sK^{-m}\bigr)$
the set $\{\xi_i\}_{i=1}^\tau$
contains a subset $\{\xi_{i_1},\xi_{i_2},\dots,\xi_{i_{m}}\}$
such that $\xi_{i_j}\in I_j$ for every admissible $j$, hence
$\xi_{i_{j+2}}\ge K\xi_{i_j}$ for any $j\le m-2$.
Conditioning on the realization of $N_s$, we obtain
by Corollary~\ref{geometric growth}:
\begin{align*}
\P&\bigl\{0\mbox{ belongs to the interior of }
\conv\{\B_n(\xi_i):\,i\le \tau\}\bigr\}\\
&\ge 1-m\exp\bigl(-sK^{-m}\bigr)-5.5\exp(-\kappa \lfloor m/2\rfloor)\\
&\ge 1-\exp(-n),
\end{align*}
and the proof is complete.
\end{proof}

The last result of this section concerns the standard random walk $W(j)$ on $\Z^n$,
which is defined as a walk with independent increments such that
each increment $W(j+1)-W(j)$ is uniformly distributed on the set $\{\pm e_j\}_{j\le n}$.
We note that the random variables $\langle \sqrt{n/m}W(m),y\rangle$ ($m\in\N$, $y\in\S^{n-1}$)
are {\it not} uniformly subgaussian;
to be more precise, their subgaussian moment depends on the dimension $n$.
At the same time, the vectors $W(m)$ still have very strong concentration properties
as the next lemma shows:
\begin{lemma}\label{Zn: concentration of increments}
Let $W(j)$ ($j\ge 0$) be the standard walk on $\Z^n$
starting at the origin, and $m\ge n^4$ be any fixed integer.
Then the vector $X:=\sqrt{n/m}W(m)$ is isotropic
and satisfies for any $y\in\S^{n-1}$:
$$
\P\bigl\{|\langle X,y\rangle|\ge t\bigr\}
\le \exp\bigl(-2(mn)^{1/4}\bigr)+2\exp(-t^2/4),\;\;t>0.
$$
In particular, $\Exp|\langle X,y\rangle|^3\le 100$ for all $y\in\S^{n-1}$, and
$X$ has the property $\mathcal{P}(\tau,\delta)$ for some universal constants $\tau,\delta$.
\end{lemma}
\begin{proof}
The isotropicity of $X$ can be easily checked.
Fix for a moment any vector $y\in \S^{n-1}$.
The random variable $\langle X,y\rangle$ can be represented as
$$\langle X,y\rangle=\sqrt{n/m}\sum\limits_{k=1}^{m}s_k,$$
where the variables $s_1,s_2,\dots,s_m$ are i.i.d.\ and each
$$s_k:=\langle W(k)-W(k-1),y\rangle$$
is symmetrically distributed, has variance $\Exp{s_k}^2=\frac{1}{n}$
and takes values in the interval $[-1,1]$.
Applying Hoeffding's inequality to the sum $\sum\nolimits_{k=1}^m {s_k}^2$,
we get
\begin{equation}\label{Zn estimate for squares}
\P\Bigl\{\sum\limits_{k=1}^m {s_k}^2\ge \frac{2m}{n}\Bigr\}
\le \exp(-2m/n^2).
\end{equation}
Further, since $s_k$ is symmetric, the distribution of the sum $\sum\nolimits_{k=1}^m s_k$
is the same as the distribution of $\sum\nolimits_{k=1}^m r_ks_k$,
where $r_1,r_2,\dots,r_m$ are Rademacher variables jointly independent with
$s_1,s_2,\dots,s_m$. Conditioning on the values of $s_k$ and using
\eqref{Zn estimate for squares} and the Khintchine inequality, we obtain for every $t>0$:
\begin{align*}
\P&\Bigl\{\Bigl|\sum\limits_{k=1}^m s_k\Bigr|\ge m t\Bigr\}\\
&=\P\Bigl\{\Bigl|\sum\limits_{k=1}^m r_ks_k\Bigr|\ge m t\Bigr\}\\
&\le\P\Bigl\{\sum\limits_{k=1}^m {s_k}^2\ge \frac{2m}{n}\Bigr\}
+\P\Bigl\{\sum\limits_{k=1}^m {s_k}^2\le \frac{2m}{n}\mbox{ and }
\Bigl|\sum\limits_{k=1}^m r_ks_k\Bigr|\ge m t\Bigr\}\\
&\le \exp(-2m/n^2)+2\exp(-m n t^2/4).
\end{align*}
Whence, in view of the bound $m\ge n^2(m n)^{1/4}$, we get
\begin{equation}\label{Zn concentration for increments}
\P\bigl\{|\langle X,y\rangle|\ge t\bigr\}
\le \exp\bigl(-2(mn)^{1/4}\bigr)+2\exp(-t^2/4),\;\;t>0.
\end{equation}
The condition \eqref{Zn concentration for increments},
together with the bound $\|X\|\le \sqrt{mn}$, gives
$\Exp|\langle X,y\rangle|^3\le 100$. It remains to apply Lemma~\ref{lemma about negative spread}.
\end{proof}

The next lemma follows from well known concentration inequalities
for subexponential random variables (see, for example, \cite[Corollary~5.17]{MR2963170}):
\begin{lemma}\label{lemma for subexp rv}
There is a universal constant $\tilde C>0$ such that
for any $N\in\N$ and
independent centered random variables $\tilde\xi_1,\tilde\xi_2,\dots,\tilde\xi_N$, each satisfying
\begin{equation}\label{subexponential estimate for Zn}
\P\bigl\{\tilde\xi_i\ge t\bigr\}\le 3\exp(-t/4),\;\;t>0,
\end{equation}
we have
\begin{equation}\label{subexp sum est for Zn}
\P\Bigl\{\sum\limits_{i=1}^N\tilde \xi_i\ge \tilde CN\Bigr\}\le 40^{-N}.
\end{equation}
\end{lemma}

In the next result, compared to Theorem~1.2 of \cite{MR3161524},
we decrease the bound on the number of steps
$N$ of the walk on $\Z^n$ sufficient to absorb the origin with high probability.

\begin{cor}\label{random walk on Zn}
There is a universal constant
$C>0$ with the following property:
Let $n,R\in\N$, $R\ge \exp(Cn)$ and let $W(j)$, $j\ge 0$,
be the standard random walk on $\Z^n$ starting at the origin. Then
$$\P\bigl\{0\mbox{ belongs to the interior of }
\conv\{W(j):\,j=1,\dots,R\}\bigr\}\ge 1-2\exp(-n).$$
\end{cor}
\begin{proof}
{\it Definition of constants and the matrix $A$.}
Let $\tau,\delta>0$ be taken from Lemma~\ref{Zn: concentration of increments}
and $\tilde C$ ~--- from Lemma~\ref{lemma for subexp rv}.
Now, we define $K:=2\sqrt{\tilde C}$ and let $L$ and $\eta$ be taken from Theorem~\ref{nongauss matrix th}.
Finally, we define $C>0$ as the smallest positive number satisfying
$$\exp(Cn)\ge(28N)^4 \Bigl\lceil\frac{4}{\eta^2}+1\Bigr\rceil^{N}$$
for any $n\in\N$ and $N=n\lceil\max(L,4/\delta^2)\rceil$.

Fix any numbers $n>0$ and $R\ge \exp(Cn)$, and let $N:=n\lceil\max(L,4/\delta^2)\rceil$.
Further, let $t_i$ ($i=0,1,\dots,N$) be numbers from $\{0,1,\dots,R\}$,
with $t_0=0$, $t_1=(28N)^4$ and $t_{i}= \bigl\lceil\frac{4}{\eta^2}+1\bigr\rceil t_{i-1}$, $i=2,3,\dots,N$.
Denote
$$X_i:=\sqrt{n}(t_{i}-t_{i-1})^{-1/2}\bigl(W(t_{i})-W(t_{i-1})\bigr),
\;\;i=1,2,\dots,N.$$
Then the vectors are isotropic, jointly independent and,
in view of Lemma~\ref{Zn: concentration of increments}, satisfy
\begin{equation}\label{X_i concentration}
\P\bigl\{|\langle X_i,y\rangle|\ge t\bigr\}
\le \exp\bigl(-2(nt_i-nt_{i-1})^{1/4}\bigr)+2\exp(-t^2/4),\;\;t>0
\end{equation}
for all $y\in\S^{n-1}$.
We let $A$ to be the $N\times n$ random matrix with rows $X_i$.

{\it Estimating the norm of $A$.}
Let $\Net$ be a $1/2$-net on $\mathbb{S}^{n-1}$ of cardinality at most $5^n$.
Fix any $y'\in \Net$. For each $i=1,2,\dots,N$, let $\xi_i:=\langle X_i,y'\rangle^2$,
and let $\tilde\xi_i$ be its truncation at level $(nt_i-nt_{i-1})^{1/4}$, i.e.\
$$\tilde\xi_i(\omega)=\begin{cases}\xi_i(\omega),
&\mbox{if }\xi_i(\omega)\le (nt_i-nt_{i-1})^{1/4},\\0,&\mbox{otherwise.}\end{cases}$$
Note that, in view of \eqref{X_i concentration}, the variables $\tilde\xi_i$
satisfy \eqref{subexponential estimate for Zn}, and
$$\P\{\xi_i\neq\tilde\xi_i\}\le 3\exp\bigl(-(nt_i-nt_{i-1})^{1/4}/4\bigr).$$
Hence, by \eqref{subexp sum est for Zn} and the above estimate, we have
\begin{align*}
\P\{\|Ay'\|\ge \sqrt{\tilde CN}\}&=\P\Bigl\{\sum\limits_{i=1}^N\xi_i\ge \tilde CN\Bigr\}\\
&\le 40^{-n}+\P\bigl\{\xi_i\neq \tilde \xi_i\mbox{ for some }i\in\{1,2,\dots,N\}\bigr\}\\
&\le 40^{-n}+3\sum\limits_{i=1}^N\exp\bigl(-(nt_i-nt_{i-1})^{1/4}/4\bigr)\\
&\le 40^{-n}+3N\exp\bigl(-7Nn^{1/4}\bigr)\\
&\le 20^{-n}.
\end{align*}
Taking the union bound for all $y'\in\Net$ and applying the standard approximation argument,
we obtain $\|A\|\le 2\sqrt{CN}=K\sqrt{N}$ with probability at least $1-\exp(-n)$.

{\it Construction of the matrix $F$ and application of Theorem~\ref{nongauss matrix th}.}
Let $F$ be the $N\times N$ non-random lower-triangular
matrix, with the entries
$$f_{ij}=\sqrt{\frac{t_{j}-t_{j-1}}{t_i-t_{i-1}}},\;\;i\ge j.$$
Obviously, $FA$ is the matrix whose $i$-th row ($i=1,\dots,N$) is precisely
the vector
$$\sqrt{\frac{n}{t_i-t_{i-1}}}W(t_i).$$
Then, in view of the definition of $t_i$'s, we have
$$\|F-\Id\|\le \frac{\eta/2}{1-\eta/2}\le\eta.$$
Finally, applying Theorem~\ref{nongauss matrix th}, we obtain
\begin{align*}
\P&\bigl\{0\mbox{ belongs to the interior of }
\conv\{W(j):\,j=1,2,\dots,R\}\bigr\}\\
&\ge \P\bigl\{0\mbox{ belongs to the interior of }
\conv\{W(t_i):\,i=1,2,\dots,N\}\bigr\}\\
&=\P\bigl\{\rank A=n\mbox{ and }\Im(FA)\cap \R_+^n=\{0\}\bigr\}\\
&\ge 1-2\exp(-n).
\end{align*}
\end{proof}

\section{Random walks on the sphere}\label{section-sphere}


Let $n>1$ and $\theta\in(0,\pi/2)$. Here, we consider the Markov chain $W_\theta$
taking values on $\S^{n-1}$ such that the angle between two 
consecutive steps is $\theta$ i.e.\ for any $i\ge 1$ we have 
$\langle W_{\theta}(i),W_{\theta}(i+1)\rangle = \cos\theta$ a.s., and the direction from $W_\theta(i)$ to $W_\theta(i+1)$ is 
chosen uniformly at random. The latter condition means that
for any $u\in\S^{n-1}$, the distribution of $W_\theta(i+1)$ conditioned on $W_\theta(i)=u$,
is uniform on the $(n-2)$-sphere $\S^{n-1}\cap \{x\in\R^n:\,\langle x,u\rangle=\cos\theta\}$.
See \cite{MR0117795} for a study of these walks and some of their generalizations.

The question addressed in this section is how many 
steps it takes for $W_\theta$ to absorb the origin into its convex hull.
Note that the answer does not depend on the distribution of the first vector $W_\theta(1)$,
and we shall further assume that $W_\theta(1)$ is uniformly distributed on the sphere. 
The question can be equivalently reformulated 
as a problem of estimating $\pi/2$-covering time of $W_\theta$. For $\phi\in(0,\pi/2]$, a $\phi$-\textit{covering} of $\S^{n-1}$ 
is any subset $S$ of the sphere such that the geodesic distance from any point of the sphere to $S$ is at most $\phi$.
Then the $\phi$-covering time for $W_\theta$ is the random variable
$$T=\min\bigl\{N:\,\mbox{the set }\{W_\theta(i),\,i\le N\}\mbox{ is a $\phi$-covering of $\S^{n-1}$}\bigr\}.$$
A related problem of estimating $\phi$-covering time of the {\it spherical Brownian motion} was considered
in \cite{MR920264} and \cite{MR3161524}, for $\phi\to 0$ and $\phi=\pi/2$, respectively.
It is not clear whether the argument developed in \cite{MR3161524} can be adopted to the walks $W_\theta$.
Our approach to the above problem is based on the results of 
Section~\ref{section-random-matrix-escape} and is completely different from the argument in \cite{MR3161524}.

The walk $W_\theta$ can be constructively described as follows:
Let $Y_1,Y_2,\dots$ be a sequence of independent standard Gaussian vectors in $\R^n$.
Let $\beta_1:=\|Y_1\|$ and define
$$W_\theta(1):=\frac{Y_1}{\| Y_1\|}=\frac{Y_{1}}{\beta_{1}}.$$
Further, for any $i\ge 1$ let
\begin{equation}\label{def-walk-sphere}
W_\theta(i+1):= \frac{\alpha_{i+1}W_\theta(i)+Y_{i+1}}{\beta_{i+1}},
\end{equation}
where 
\begin{equation}
\begin{split}
\beta_{i+1}&:=\Vert \alpha_{i+1} W_{\theta}(i)+Y_{i+1}\Vert \quad \text{ and }\\
\alpha_{i+1}&:=\cot \theta\, \Vert P_i Y_{i+1}\Vert-\langle Y_{i+1}, W_{\theta}(i)\rangle,\;\;i\ge 1,\label{def-beta-alpha}
\end{split}
\end{equation} 
with $P_i$ denoting the (random) orthogonal projection onto the hyperplane orthogonal to $W_{\theta}(i)$.
It can be easily checked that
$$\beta_{i}=\frac{\|P_{i-1} Y_{i}\|}{\sin\theta},\;\;i\ge 2,$$
and that $W_{\theta}$ is the Markov process described at the beginning of the section.
For any $i=2,3,\dots$ the coefficients $\alpha_i$ and $\beta_i$ are random 
variables depending on $Y_i$ and $W_{\theta}(i-1)$. Using \eqref{feller ineq} and \eqref{gaussian norm ineq}, one can 
deduce the following concentration inequalities:
\begin{lemma}\label{concentration-alpha-beta}
There exist a universal constant $c>0$ such that for
$\delta_\theta:= c\min(1,\cot\theta)$ and for any $i=2,3,\dots$ and $\varepsilon>0$ we have
$$
\P\bigl\{(1-\varepsilon)\sqrt{n}\cot\theta\le \alpha_i\le (1+\varepsilon)\sqrt{n}\cot\theta\bigr\}
\ge 1-2\exp(-{\delta_{\theta}}^2 \varepsilon^2n)
$$
and
$$
\P\bigl\{(1-\varepsilon)\sin\theta/\sqrt{n}\le {\beta_i}^{-1}\le (1+\varepsilon)\sin\theta /\sqrt{n}\bigr\}
\ge 1-2\exp(-{\delta_{\theta}}^2 \varepsilon^2n).
$$
\end{lemma}
\noindent
Moreover, \eqref{gaussian norm ineq} immediately implies
\begin{equation}\label{beta1 concentration}
\P\bigl\{(1-\varepsilon)/\sqrt{n}\le {\beta_1}^{-1}\le (1+\varepsilon)/\sqrt{n}\bigr\}
\ge 1-2\exp(-c \varepsilon^2n),\;\;\varepsilon>0,
\end{equation}
provided that the constant $c$ is sufficiently small.
Before we state the main result of the section, let us consider the following elementary lemma:
\begin{lemma}\label{calculation}
For any $q\in(0,1)$ and $0<\varepsilon\le\frac{1-q}{8}$ we have
$$\sum\limits_{k=0}^\infty \bigl((1+\varepsilon)^{2k+1}-1\bigr)q^k\le\frac{4\varepsilon}{(1-q)^2}.$$
\end{lemma}
\begin{proof}
First, note that the conditions on $\varepsilon$ and $q$ imply
$$q(1+\varepsilon)^2
\le \frac{81q}{64}-\frac{9q^2}{32}+\frac{q^3}{64}
\le q+\frac{17q}{64}-\frac{17q^2}{64}\le\frac{1+q}{2},$$
whence
$$1-q(1+\varepsilon)^2\ge \frac{1-q}{2}.$$
Using the last inequality, we obtain
\begin{align*}
\sum\limits_{k=0}^\infty \bigl((1+\varepsilon)^{2k+1}-1\bigr)q^k
&=(1+\varepsilon)\sum\limits_{k=0}^\infty \bigl(q(1+\varepsilon)^2 \bigr)^k
-\sum\limits_{k=0}^\infty q^k\\
&=\frac{(1+\varepsilon)}{1-q(1+\varepsilon)^2}-\frac{1}{1-q}\\
&=\frac{\varepsilon+\varepsilon q+\varepsilon^2 q}{(1-q)(1-q(1+\varepsilon)^2)}\\
&\le\frac{4\varepsilon}{(1-q)^2}.
\end{align*}
\end{proof}

\begin{theor}\label{th-random-walk-sphere}
For any $\theta\in (0,\pi/2)$ there exist $n_0=n_0(\theta)$ and $K=K(\theta)$ depending 
only on $\theta$ such that the following holds: Let $n\ge n_0$ and let $W_\theta$ be
the random walk on $\S^{n-1}$ defined above. Then for all $N\ge Kn$ we have 
$$
\P\bigl\{0\mbox{ belongs to }\conv\{W_{\theta}(i):\,i\le N\}\bigr\}\ge 1-\exp(-n).
$$ 
\end{theor}
\begin{proof}
Fix an angle $\theta\in(0,\pi/2)$.
Let $\gamma:=\frac{\sin\theta\,(1-\cos\theta)}{1+\cos\theta}$
and let $\eta, L$ and  $\kappa$ be as in Theorem~\ref{gaussian-escape}. Define
$\varepsilon:=\eta\sin\theta\,(1-\cos\theta)^2/4$ and let
$n_0$ be the smallest integer such that for all $n\ge n_0$ we have
$$5.5\exp ( -\kappa \lceil Ln\rceil )+4\lceil Ln\rceil\exp(-{\delta_{\theta}}^2 \varepsilon^2n)\le \exp(-\mu\, n),$$
where $\mu=\frac{1}{2}\min\bigl(\kappa,{\delta_{\theta}}^2 \varepsilon^2)$
and $\delta_{\theta}$ is taken from Lemma~\ref{concentration-alpha-beta}.

Fix $n\ge n_0$. First, we show that $\tilde N:=\lceil Ln\rceil$ steps is sufficient to get the origin
in the convex hull of $W_\theta(i)$ ($i\le\tilde N$) with probability $1-\exp(-\mu\, n)$.
This shall be done by using the representation \eqref{def-walk-sphere} for the walk $W_\theta$
and by applying Theorem~\ref{gaussian-escape}. Then we will augment the probability
estimate to $1-\exp(-n)$ by increasing the number of steps.

Let $G$ be the standard $\tilde N\times n$ Gaussian matrix
with rows $Y_i$ ($i\le \tilde N$).
We shall construct a random lower-triangular $\tilde N\times \tilde N$ matrix $F$ such that the $i$-th row of $FG$ is $W_\theta(i)$.
Define $F:=(f_{ij})$ with
$$
f_{ij}:=\frac{\prod_{k=j+1}^{i}\alpha_k}{\prod_{k=j}^i \beta_k}\; \text{ for } j< i\le\tilde N
\quad\text{ and }\quad f_{ii}:= \frac{1}{\beta_i}\; \text{ for } i\le\tilde N,
$$
where $\alpha_k$ and $\beta_k$ are given by \eqref{def-beta-alpha}.
Since $FG=(W_\theta(1),W_\theta(2),\dots,W_\theta(\tilde N))^t$, the origin
does not belong to $\conv\{W_\theta(i):\,i\le \tilde N\}$ only if
there exists $y\in\S^{n-1}$ such that $FGy\in \R_+^{\tilde N}$. Now define 
$\tilde F$ as the $\tilde N\times \tilde N$ lower triangular matrix whose entries are given by 
$$
\tilde f_{i1}=\frac{\left(\cos\theta\right)^{i-1}}{\sqrt{n}}\;\;\text{for any } i\le \tilde N
\;
\text{ and }\; \tilde f_{ij}:=\sin\theta\,\frac{\left(\cos\theta\right)^{i-j}}{\sqrt{n}}\;\;\text{for } 2\le j\le i.
$$
It is not difficult to see that
\begin{equation}\label{norm-F-tilde}
\frac{\sin\theta}{\sqrt{n}}\le\|\tilde F\|\le \frac{1}{(1-\cos \theta)\sqrt{n}}.
\end{equation}
Further,
let $Q$ be the matrix obtained from $\tilde F$ by multiplying the first column of $\tilde F$ by $\sin\theta$
and leaving the other columns unchanged. Then, clearly, $s_{\min}(Q)\le s_{\min}(\tilde F)$ implying $\|\tilde F^{-1}\|\le\|Q^{-1}\|$.
On the other hand, the inverse of $Q$ is a lower bidiagonal matrix with $\frac{\sqrt{n}}{\sin\theta}$
on the main diagonal and $-\cos\theta\frac{\sqrt{n}}{\sin\theta}$ on the diagonal below.
Hence, $\|\tilde F^{-1}\|\le\|Q^{-1}\|\le (1+\cos\theta)\frac{\sqrt{n}}{\sin\theta}$,
and the condition number of $\tilde F$ satisfies
$$\|\tilde F\|\cdot\|\tilde F^{-1}\|\le \frac{1+\cos\theta}{\sin\theta\,(1-\cos\theta)}=\gamma^{-1}.$$
Applying Theorem~\ref{gaussian-escape}, we get
\begin{equation*}\label{eq-sphere-part1}
\P\bigl\{\exists y\in \mathbb{S}^{n-1}, \; FGy\in \R^{\tilde N}_{+}\bigr\}
\le    5.5\exp ( -\kappa \tilde N )+ \P\bigl\{ \bigl\Vert F-\tilde F\bigr\Vert > \eta\Vert \tilde F\Vert\bigr\}.
\end{equation*}
It remains to bound the probability $\P\bigl\{ \bigl\Vert F-\tilde F\bigr\Vert > \eta\Vert \tilde F\Vert\bigr\}$.
In view of Lemma~\ref{concentration-alpha-beta} and \eqref{beta1 concentration},
with probability at least $1-4\tilde N\exp(-{\delta_{\theta}}^2 \varepsilon^2n)$ we have
$$
\bigl| f_{ij}-\tilde f_{ij}\bigr| \le \left( (1+\varepsilon)^{2(i-j)+1}-1\right)\tilde f_{ij}\;\;\mbox{for any }j\le i. 
$$
This, together with Lemma~\ref{calculation} 
and \eqref{norm-F-tilde}, 
implies that   
$$
\Vert F-\tilde F\Vert \le \frac{1}{\sqrt{n}} \sum\limits_{k=0}^\infty \bigl((1+\varepsilon)^{2k+1}-1\bigr)(\cos\theta)^k
\le \frac{4\varepsilon}{(1-\cos\theta)^2\sqrt{n}}\le \eta \Vert \tilde F\Vert
$$
with probability at least $1-4\tilde N\exp(-{\delta_{\theta}}^2 \varepsilon^2n)$.
Hence, by the restriction on $n_0$,
$$\P\bigl\{\exists y\in \mathbb{S}^{n-1}, \; FGy\in \R^{\tilde N}_{+}\bigr\}
\le 5.5\exp ( -\kappa \tilde N )+4\tilde N\exp(-{\delta_{\theta}}^2 \varepsilon^2n)\le \exp(-\mu n),$$
where $\mu=\frac{1}{2}\min\bigl(\kappa,{\delta_{\theta}}^2 \varepsilon^2)$.
Finally, if $N\ge \lceil\mu^{-1}\rceil\tilde N$ then the above estimate implies
\begin{align*}
\P&\bigl\{0\mbox{ does not belong to }\conv\{W_{\theta}(i):\,i\le N\}\bigr\}\\
&\le \P\bigl\{0\mbox{ does not belong to }\conv\{W_{\theta}(i):\,i\le \tilde N\}\bigr\}^{\lceil\mu^{-1}\rceil}\\
&\le \exp(-n).
\end{align*}
\end{proof}

\section{Minimax of the $n$-dimensional Brownian motion}\label{section-minimax}

In this section we will prove Theorem~B which, as noted in the introduction, is equivalent to estimate \eqref{eq-minmax-intro}. 
 
Let us give an informal description of the proof. We construct a random unit vector 
$\bar{v}$ in $\R^n$ such that with probability close to one 
\begin{equation}\label{eq-intro-goal}
\langle \bar{v}, \B_n(t)\rangle >0 \quad \text{for any }  t\in [1,2^{cn}].
\end{equation}
The construction procedure shall be divided into a series of steps. At the initial step, we produce 
a random vector $\bar{v}_0$ such that
$$
\langle \bar{v}_0, \B_n(2^i)\rangle >0 \quad \text{for any } i=0,1,\dots, cn.
$$
(In fact, $\bar{v}_0$ will satisfy a stronger condition). At a step $k$, $k\ge 1$, 
we update the vector $\bar{v}_{k-1}$ by adding a small perturbation in such a way that 
$$
\langle \bar{v}_{k}, \B_n(2^{j2^{-k}})\rangle >0  \quad \text{for any } j=0,1,\dots, 2^kcn.
$$
(Again $\bar{v}_k$ will in fact satisfy a stronger condition). Finally, using some standard properties of the Brownian 
bridge, we verify that 
$\bar{v}:= \bar{v}_{\log_2\ln n}$ satisfies \eqref{eq-intro-goal} with a large probability.

\subsection{\bf Auxiliary facts}
\label{sec:1}

In this subsection we introduce several auxiliary results that will be used within the proof.
The proof of the next lemma is straightforward, so we omit it.
\begin{lemma}\label{brownian-segment-split}
Let $\B_n(t)$ ($0\le t<\infty$) be the standard Brownian motion in $\R^n$ and let $0<a<b$.
Fix any $s\in(a,b)$ and set
$$
w(s):= \frac{b-s}{b-a}\B_n(a)+ \frac{s-a}{b-a}\B_n(b);\;\;u(s):= \B_n(s)-w(s).
$$ 
Then the process $u(s)$, $s\in(a,b)$, is a Brownian bridge, and
\begin{enumerate}
\item $u(s)$ is a centered Gaussian vector with the covariance matrix
$$\frac{(b-s)(s-a)}{b-a}\Id_n.$$
\item The random vector $u(s)$ is independent from the process $\B_n(t)$ indexed over $t\in(0,a]\cup[b,\infty)$.
\end{enumerate}
\end{lemma}

\begin{lemma}\label{normal-vector-lem}
Let $d,m \in\mathbb{N}$ be such that $m\leq d/2$. Let $X_1, X_2,\ldots, X_m$ be independent standard Gaussian 
vectors in $\mathbb{R}^d$. Then for any non-random vector $b\in\Sphere^{m-1}$, there exists a random unit vector $\bar{u}_b\in\R^d$ 
such that 
$$
\P \Big\{ \langle \bar{u}_b,X_i\rangle \ge c_{\ref{normal-vector-lem}}\sqrt{d}
\vert b_i\vert \text{ for all } i=1,2,\dots ,m\Big\}\ge 1-\exp(-c_{\ref{normal-vector-lem}}d),
$$
where $c_{\ref{normal-vector-lem}}$ is a universal constant and $b_i$'s are the coordinates of $b$.
Moreover, $\bar u_b$ can be defined as a Borel function of $X_i$'s and $b$.
\end{lemma}
\begin{proof}
Without loss of generality, we can assume that $b_i\neq 0$ for any $i\le m$
and that $X_i$'s are linearly independent on the entire probability space. 
Denote by $E$ the random affine subspace of $\R^d$ spanned by $\{\vert b_i\vert^{-1}X_i\}_{i\le m}$.
Define $\bar{u}_b$ as the unique unit vector in $\spn\{X_1,\dots,X_m\}$ such that $\bar{u}_b$ is orthogonal to 
$E$ and for any $i\le m$ we have
$$\langle \bar{u}_b, \vert b_i\vert^{-1}X_i \rangle = \dist(0, E),$$
where $\dist(0, E)$ stands for the distance from the origin to $E$. Then we have
\begin{equation}\label{eq-dist-E}
\sum_{i\le m} \langle \bar{u}_b, X_i\rangle^2= \sum_{i\le m} \bigl\langle \bar{u}_b,
\frac{X_i}{\vert b_i\vert}\bigr\rangle^2 \vert b_i\vert^2 
= \sum_{i\le m} \dist(0, E)^2\cdot \vert b_i\vert^2= \dist(0, E)^2.
\end{equation}
Let $G$ be the $d\times m$ standard Gaussian matrix with columns $X_i$, $i=1,2,\ldots, m$. 
Using the definition of $\bar{u}_b$ together with \eqref{eq-dist-E}, we obtain for any $\tau>0$:
\begin{align*}
\P \Big\{ \langle \bar{u}_b,X_i\rangle \ge \tau\sqrt{d} \vert b_i\vert \text{ for all } i=1,2,\dots ,m\Big\}
&= \P \Big\{ \dist(0,E)\ge \tau\sqrt{d}\Big\}\\
&= \P \Big\{ \sqrt{\sum_{i\le m} \langle \bar{u}_b, X_i\rangle^2}\ge \tau\sqrt{d}\Big\}\\
&=\P\Big\{\Vert G^t\bar{u}_b\Vert\ge \tau\sqrt{d}\Big\}\\
&\ge \P\Big\{ s_{\min}(G)\ge \tau\sqrt{d}\Big\},
\end{align*}
where the last inequality holds since $\bar{u}_b\in {\rm Im}\, G$. 
The proof is finished by choosing a sufficiently small $c_{\ref{normal-vector-lem}}:=\tau$
and applying \eqref{gaussian matrix ineq}.
\end{proof}

\begin{lemma}\label{norm of noncentered gaussian}
Let $q\in\N$ and $r\in\R$ with $e\le r\le \sqrt{\ln q}$, and let
$g_1, g_2,\ldots, g_q$ be independent standard Gaussian variables. 
Define a random vector $b=(b_1,b_2,\dots,b_q)\in\R^q$
by $b_i:=\max(0, g_i - r)$, $i\le q$. Then 
$$
\P\Big\{ \Vert b\Vert \le 4\sqrt{q}\exp(-r^2/8)\Big\} \ge 1-\exp(-2\sqrt{q}).
$$ 
\end{lemma}
\begin{proof}
Let $\lambda \in (0,1/2)$. We have 
$$
\E e^{\lambda \Vert b\Vert^2} = \prod_{i=1}^q \E e^{\lambda {b_i}^2}
= \left( 1+\int_{1}^\infty \P\{ e^{\lambda {b_1}^2}\ge \tau\} d\tau\right)^q.
$$
Next, using \eqref{feller ineq}, we get
\begin{align*}
\int_{1}^\infty \P\{ e^{\lambda b_1^2}\ge \tau\} d\tau 
&\le (r-1)\P\{g_1>r\} + \int_{r}^{\infty} \P\{e^{\lambda b_1^2}\ge \tau\} d\tau\\
&\le e^{-r^2/2} +  \int_{r}^{\infty} \P\Bigl\{g_1\ge \sqrt{\frac{\ln\tau}{\lambda}}\Bigr\} d\tau\\
&\le e^{-r^2/2}+ \int_{r}^{\infty} \tau^{-\frac{1}{2\lambda}}d\tau\\
&=e^{-r^2/2}+ \frac{r^{1-\frac{1}{2\lambda}}}{\frac{1}{2\lambda}-1}.
\end{align*}
Now, take $\lambda=\bigl(2+\frac{r^2}{\ln r}\bigr)^{-1}$ so that $\frac{1}{2\lambda} -1= \frac{r^2}{2\ln r}$. 
After replacing $\lambda$ with its value, we deduce that 
\begin{equation}\label{eq-laplace}
\E e^{\lambda \Vert b\Vert^2} \le \bigl( 1+2e^{-r^2/2}\bigr)^q\le \exp(2qe^{-r^2/2}).
\end{equation}
Using Markov's inequality together with \eqref{eq-laplace}, we obtain
$$
\P\{\lambda \Vert b\Vert^2\ge 4q e^{-r^2/2}\} \le \exp(-2qe^{-r^2/2})\le \exp(-2\sqrt{q}),
$$
where the last inequality holds since $r\le \sqrt{\ln q}$. To finish the proof, it remains to note that
$$
\frac{4q e^{-r^2/2}}{\lambda} \le 8qr^2e^{-r^2/2}\le 16q e^{-r^2/4}.
$$
\end{proof}

\subsection{\bf Proof of Theorem~B}\label{the proof section}

Throughout this part, we assume that $c>0$ and $n_0\in\N$ are appropriately chosen constants (with $c$ sufficiently
small and $n_0$ sufficiently large) and $n\ge n_0$ is fixed. The admissible values for
$c$ and $n_0$ can be recovered
from the proof, however, we prefer to avoid these technical details.
Further, in order not to overload the presentation,
from now on we treat certain real-valued parameters are integers.
In particular, this concerns the product $cn$,
as well as several other quantities depending on $n$ (we will point them out later).
To prove relation \eqref{eq-minmax-intro}, we will construct a random unit vector $\bar v\in\R^n$ such that
\begin{equation}\label{condition on barn}
\langle \bar v,\B_n(t)\rangle > 0\;\quad \text{for any } t\in[1,2^{cn}]
\end{equation}
with probability close to one. 

Let $N:=cn$ and define
$$a_0:=0 \text{ \ and \ }a_i:=2^{i-1},\;\;i=1,2,\dots,N+1.$$
The starting point of the proof is to define a random vector $\bar v_0$ such that  $\langle \bar v_0,\B_n(a_i)\rangle$ is 
large for all $i\leq N+1$.
For this, we will use Lemma~\ref{normal-vector-lem} taking all coordinates of the vector $b$ equal.
It will be more convenient to state the next lemma (which is a direct consequence of Lemma~\ref{normal-vector-lem})
with generic parameters $m$ and $d$ instead of $N$, $n$.

\begin{lemma}\label{lem-v0}
Let $d, m\in \N$ with $m\leq d/2$ and $\B_d(t)$ be the standard Brownian motion in $\R^d$. Then there exists a random unit vector 
$\bar v_0\in \R^d$ such that 
\begin{align*}
\P&\Bigl\{\langle \bar v_0 , \B_d(a_{i+1})-\B_d(a_i)\rangle \geq \frac{c_{\ref{normal-vector-lem}}}{2}
\sqrt{\frac{d a_{i+1} }{m} },\, i=0,\ldots m\Bigr\}\\
&\geq 1-\exp(-c_{\ref{normal-vector-lem}} d). 
\end{align*}
\end{lemma}

We note that, conditioned on a realization of $\B_d(a_1),\ldots, \B_d(a_{m+1})$ (hence, $\bar v_0$),  
for each admissible $i\geq 1$ the process
$$\langle\bar v_0, \B_d\big(a_i +t(a_{i+1}-a_i)\big)\rangle,\;\; t\in[0,1],$$ 
is a (non-centered) Brownian bridge, and standard estimates (see, for example, \cite[p.~34]{SW}) 
together with above lemma imply that 
given $i$, we have $\langle \bar v_0 , \B_d(a_i +t(a_{i+1}-a_i))\rangle > 0$ for all $t\in[0,1]$ with 
probability at least $1-2\exp(-c''d/m)$ for a universal constant $c''$. If $m\ll d/\ln d$ then applying the union 
bound we get $\langle \bar v_0,\B_d(t)\rangle >0$ for all $1\leq t\leq a_{m+1}$ with high probability. 

The argument described above is given in \cite{MR3161524}. Note that for $m\gg d/\ln d$ the probability that the $i$-th Brownian bridge
is not positive becomes too large to apply the union bound over all $i$. For this reason, we significantly modified the approach
of \cite{MR3161524}.
Let $M:=\log_2 \ln n$ (we will further treat the quantity as an integer, omitting a truncation operation).
Our construction will be iterative: after defining vector $\bar v_0$ as 
described above, we will produce a sequence
of random vectors $\bar v_k$, $k=1,\dots,M$, where each $\bar v_k$ with a high probability
satisfies $\langle \bar v_k,\B_n(t)\rangle >0$ for all $t$ in a certain discrete subset of $[1,2^{cn}]$.
The subset for $\bar v_k$ is obtained by zooming in and adding mid-points 
between every two neighbouring points of the subset generated for $\bar v_{k-1}$.
The size of those discrete subsets grows with $k$ exponentially, so
that the vector $\bar v:=\bar v_{M}$ will possess the required property \eqref{condition on barn} with probability
close to one. The definition of the subsets is made more precise below.

We split the interval $[0,a_{N+1}]$ into blocks.
For each admissible $i\ge 0$,
{\it the $i$-th block} is the interval $[a_i,a_{i+1}]$.
With the $i$-th block, we associate a sequence of sets
$I^i_k$, $k=0,1,\dots,M,$ in the following way:
for $i=0$ we have $I^i_k=\emptyset$ for all $k\ge 0$;
for $i\ge 1$, we set $I^i_0=\emptyset$ and
$$I^i_k:=\{2^{1/2^{k}}a_i,2^{2/2^k}a_i,2^{3/2^k}a_i,\dots,2^{(2^k-1)/2^k}a_i\},
\;\;k=1,2,\dots,M.$$

Given any $0<k\le M$, the vector $\bar v_k$ will be a small perturbation of the vector $\bar v_{k-1}$.
The operation of constructing $\bar v_k$ will be referred to as {\it the $k$-th step} of the construction.
We must admit that the construction is rather technical. In fact, each step itself is divided
into a sequence of {\it substeps}. To make the exposition of the proof as clear as possible, we
won't provide all the details at once but instead introduce them sequentially.

At each step, to avoid issues connected with probabilistic dependencies, the already constructed vector $\bar v_{k-1}$ and
the perturbation added to it will be defined on disjoint coordinate subspaces of $\R^n$.
Namely, we split $\R^n$ into $M+1$ coordinate subspaces as follows
$$
\R^n:=\prod_{k=0}^{M}\R^{J^k},
$$
where $J^k$ are pairwise disjoint subsets of $\{1,\dots,n\}$ with
$|J^k|=\tilde c n2^{-k/8}$ for an appropriate constant $\tilde c$ (chosen so that $\sum_{k\leq M} \vert J^k\vert =n$) and
$\R^{J^k}= \spn\{e_i\}_{i\in J^k}$. Again, for a lighter exposition
we treat the quantities $\tilde c n2^{-k/8}$ as integers. For every $k\le M$, define $\Proj^k:\R^n\to \R^n$
as the orthogonal projection onto $\R^{J^k}$.

Let $F, H: \N\to\R_+$ be a decreasing and an increasing function, respectively, satisfying the relations
\begin{equation}\label{property-F-H}
8c\, F(1)^2=\tilde c\, {c_{\ref{normal-vector-lem}}}^2\quad \text{ and }\quad  \forall k\leq M,\quad F(k)\geq C_f\geq 2H(k) ,
\end{equation}
where $C_f>0$ is a constant which will be determined later. 

Now, we can state more precisely what we mean by the $k$-th step of the construction ($k=0,1,\dots,M$). {\bf The goal
of the $k$-th step is to produce a random unit vector $\bar v_k$ with the following properties:}
\begin{align}
&\begin{aligned}{\bf 1.}\;\bar v_k\mbox{ is supported on }\prod_{p=0}^{k}\R^{J^p};\end{aligned}\label{nk prop 1}\\
&\begin{aligned}&{\bf 2.}\;\bar v_k\mbox{ is measurable with respect to the $\sigma$-algebra generated by}\\
&\mbox{$\Proj^p (\B_n(t))$ for all $0\le p\le k$, $t\in\bigcup_{i=0}^N\bigl(\{a_{i+1}\}\cup I^i_k\bigr)$};
\end{aligned}\label{nk prop 2} \\
&\vphantom{A^{\int\limits^1}}{\bf 3.}\;\mbox{The event}\nonumber\\
&\hspace{1cm}\begin{aligned}
\Event_k:=\Bigl\{&\langle \bar v_{k},\B_n(t)-\B_n(a_i)\rangle\ge -H(k+1)\sqrt{a_i}\mbox{ and}\Bigr.\\
\Bigl.&\langle \bar v_{k},\B_n(a_{i+1})-\B_n(a_{i})\rangle\ge F(k+1)\sqrt{a_{i+1}}\Bigr.\\
\Bigl.&\mbox{for all }t\in I^i_k\mbox{ and }i=0,1,\dots,N\Bigr\}
\end{aligned}\nonumber\\
&\mbox{has probability close to one}\nonumber.
\end{align}

Quantitative estimates of $\P(\Event_k)$ are provided by the following lemma which will be proved in the next section.
\begin{lemma}[$k$-th Step]\label{k step lemma}
For a small enough constant $c>0$ and a large enough $C_f>0$, there exist $F$ and $H$ satisfying (\ref{property-F-H}) 
such that the following holds. Let $1\le k\le M$ and assume that a random unit vector $\bar v_{k-1}$ satisfying properties \eqref{nk prop 1}, \eqref{nk prop 2}
has been constructed. Then there exists a random unit vector $\bar v_{k}$ satisfying \eqref{nk prop 1}---\eqref{nk prop 2}
and such that
$$\P(\Event_k)\ge\P(\Event_{k-1})-\frac{1}{n^2}.$$
\end{lemma}

\begin{proof}[Proof of Theorem~B]
In view of the relation \eqref{property-F-H},
we have 
$$
2F(1)= c_{\ref{normal-vector-lem}}\sqrt{\frac{\tilde c}{2c}} \leq c_{\ref{normal-vector-lem}}\sqrt{\frac{\vert J^0\vert}{N}}.
$$
Hence, in view of Lemma~\ref{lem-v0} (applied with $m=N$ and $d=\vert J^0\vert$), there exists a random unit vector $\bar v_0\in\R^{J^0}$ measurable
with respect to the $\sigma$-algebra generated by vectors $\Proj^0(\B_n(a_{i+1})-\B_n(a_i))$, $i=0,1,\dots,N$,
and such that
\begin{align*}
\P(\Event_0)&=
\P\bigl\{\langle \bar v_0,\B_n(a_{i+1})-\B_n(a_i)\rangle \ge F(1)\sqrt{a_{i+1}}\mbox{ for }
i=0,1,\dots,N\bigr\}\\
&\ge 1-\exp(-c_{\ref{normal-vector-lem}}|J^0|)\\
&\ge 1-\frac{1}{n^2}.
\end{align*}
Applying Lemma~\ref{k step lemma} $M$ times, we obtain a random unit vector $\bar v_M$ satisfying
\eqref{nk prop 1}--\eqref{nk prop 2} such that
$$\P(\Event_M)\ge 1-\frac{M+1}{n^2}.$$
Note that everywhere on $\Event_M$, we have
$$\langle \bar v_M,\B_n(a_{i+1})\rangle\ge \langle \bar v_M,\B_n(a_{i+1})-\B_n(a_{i})\rangle\ge C_f\sqrt{a_{i+1}}$$
and
$$\langle \bar v_M,\B_n(t)\rangle \ge \langle \bar v_M,\B_n(a_i)\rangle-\frac{C_f}{2}\sqrt{a_i}\ge\frac{C_f}{2}\sqrt{a_i},\;\;
t\in I^i_k$$
for all $i=0,1,\dots,N$.
Hence, denoting $Q:=\{a_1,a_2,\dots,a_{N+1}\}\cup\bigcup_{i=1}^N I^i_M$, we get
\begin{equation}\label{eventM is subset}
\Event_M\subset\Bigl\{\bigl\langle\bar v_M,\frac{\B_n(t)}{\sqrt{t}}\bigr\rangle\ge \frac{C_f}{4},\;\;t\in Q\Bigr\}.
\end{equation}
Now, take any two neighbouring points $t_1<t_2$ from $Q$.
Note that, conditioned on a realization of vectors $\B_n(t)$, $t\in Q$, the random process
$$X(s)=
\bigl\langle\bar v_M,\frac{s\B_n(t_2)+(1-s)\B_n(t_1)}{\sqrt{t_2-t_1}}\bigr\rangle
-\bigl\langle \bar v_M,\frac{\B_n(t_1+s(t_2-t_1))}{\sqrt{t_2-t_1}}\bigr\rangle,$$
defined for $s\in[0,1]$,
is a standard Brownian bridge. Hence (see, for example, \cite[p.~34]{SW}), we have for any $\tau>0$
$$\P\bigl\{X(s)\ge \tau\mbox{ for some }s\in [0,1]\bigr\}=\exp(-2\tau^2).$$
Taking $\tau:=2\sqrt{\ln n}$, we obtain
\begin{align*}
\P\bigl\{\bigl\langle\bar v_M,\B_n(t)\bigr\rangle
&\le\min\bigl(\langle\bar v_M,\B_n(t_1)\rangle,\langle\bar v_M,\B_n(t_2)\rangle\bigr)\bigr.\\
&\bigl.-2\sqrt{t_2-t_1}\sqrt{\ln n}\mbox{ for some }t\in[t_1,t_2]\bigr\}\\
&\hspace{-2cm}\le\frac{1}{n^8}.
\end{align*}
Finally, note that, in view of \eqref{eventM is subset}, everywhere on $\Event_M$ we have
\begin{align*}
&(t_2-t_1)^{-1/2}\min\bigl(\langle\bar v_M,\B_n(t_1)\rangle,\langle\bar v_M,\B_n(t_2)\rangle\bigr)-2\sqrt{\ln n}\\
&\ge \frac{C_f}{4}\sqrt{\frac{t_1}{t_2-t_1}}-2\sqrt{\ln n}\\
&\ge 2^{M/2-3}C_f-2\sqrt{\ln n}\\
&>0.
\end{align*}
Taking the union bound over all adjacent pairs in $Q$ (clearly, $|Q|\le n^2$), we come to the relation
$$\P\bigl\{\langle\bar v_M,\B_n(t)\rangle>0\mbox{ for all }t\in[1,2^{cn}]\bigr\}\ge\P(\Event_M)-\frac{|Q|}{n^8}\ge 1-\frac{1}{n}.$$
\end{proof}

\subsection{\bf Proof of Lemma~\ref{k step lemma}}

Let $M'=\frac{1}{4}\log_2 \ln n$. For every $k\leq M$, we split $J^k$ into pairwise 
disjoint subsets $J_\ell^k$, $\ell \leq M'$, with $|J_\ell^k|=c'n2^{-(k+\ell)/8}$ 
for an appropriate constant $c'$, chosen so that $\sum_{\ell\leq M'}\vert J_\ell^k\vert=\vert J^k\vert$ 
(to make computations lighter, we will treat the quantities $c'n2^{-(k+\ell)/8}$, $k\leq M,\ell\leq M'$,
as integers).
For every $k\le M,\ell\le M'$, define $\Proj_\ell^k:\R^n\to \R^n$
as the orthogonal projection onto $\R^{J_\ell^k}$.

Further, we define two functions $f,h:\N\times\N_0\to\R_+$
as follows:
\begin{enumerate}
\item $f$ is decreasing in both arguments;
$f(1,0)=C_f+2^{-1/2}(1-2^{-1/4})^{-2} C_f$; for each $k> 0$ and
$\ell>0$ we have $f(k,\ell-1)-f(k,\ell)=C_f2^{-(k+\ell)/4}$;
finally, $f(k,0)=\lim\limits_{\ell\to\infty}f(k-1,\ell)$ for all $k> 1$.
The constant $C_f>0$ is defined via the relation $8c f(1,0)^2=\tilde c {c_{\ref{normal-vector-lem}}}^2$,
where $\tilde c$ is taken from the definition of sets $J^k$ and $c_{\ref{normal-vector-lem}}$
comes from Lemma~\ref{normal-vector-lem}.

\item $h$ is increasing in both arguments;
$h(1,0)=0$; for each $k> 0$ and
$\ell>0$ we have $h(k,\ell)-h(k,\ell-1)=C_h2^{-(k+\ell)/4}$;
moreover, $h(k,0)=\lim\limits_{\ell\to\infty}h(k-1,\ell)$ for all $k> 1$.
The constant $C_h$ is defined by $C_h=2^{-1/2}(1-2^{-1/4})^{2}C_f$.
\end{enumerate}

Now define $F:\N\to \R$ and $H:\N\to \R$ by $F(k):=f(k,0)$ and $H(k):=h(k,0)$ for any $k\in\N$.   
Note that $F$ and $H$ satisfy \eqref{property-F-H}.

Fix $k\ge 1$. Assuming that the vector $\bar v_{k-1}$ is already constructed, the aim is to construct $\bar v_k$  such that 
the event $\Event_k$ has large probability. 
The vector $\bar v_k$ is obtained via an embedded iteration procedure realized as
a sequence of substeps. Namely, we set $\bar v_{k,0}:=\bar v_{k-1}$ and inductively construct random vectors $\bar v_{k,\ell}$, $1\le\ell\le M'$ 
and take $\bar v_k=\bar v_{k,M'}$. Let us give a partial description of the procedure, omitting some details.

For each $\ell=1,2,\dots,M'+1$ and every
block $i=0,1,2,\dots,N$ {\it the $i$-th block statistic} is
\begin{equation}\label{def-statistics}
\begin{split}
\BStat_i(k,\ell):=\max\Bigl(0,
&\max\limits_{t\in I^i_k}\bigl\langle \bar v_{k,\ell-1},
\frac{\B_n(a_i)-\B_n(t)}{\sqrt{a_i}}\bigr\rangle-h(k,\ell),\Bigr.\\
\Bigl.&\bigl\langle \bar v_{k,\ell-1},\frac{\B_n(a_i)-\B_n(a_{i+1})}{\sqrt{a_{i+1}}}\bigr\rangle+f(k,\ell)\Bigr).
\end{split}
\end{equation}
Note that the statistic for the zero block is simply
$$\max\Bigl(0,-\bigl\langle \bar v_{k,\ell-1},
\B_n(a_{1})\bigr\rangle+f(k,\ell)\Bigr).$$
The $(N+1)$-dimensional vector $\bigl(\BStat_0(k,\ell),\dots,\BStat_N(k,\ell)\bigr)$
will be denoted by $\BStat(k,\ell)$. Let us also denote
$$\BadBlocks(k,\ell):=\bigl\{i:\,\BStat_i(k,\ell)\neq 0\bigr\}.$$
Note that the event $\{\BadBlocks(k,M'+1)=\emptyset\}$ is contained inside $\Event_k$. 
At each substep, using information about the statistics 
$\BStat(k,\ell)$ and choosing an appropriate perturbation of $\bar v_{k,{\ell-1}}$ 
to obtain $\bar v_{k,\ell}$, we will control the measure of the event $\{\BadBlocks(k,\ell+1)=\emptyset\}$,
and in this way will be able to estimate the probability of $\Event_k$ from below.

{\bf Given $\bar v_{k,\ell-1}$, the goal of the $\ell$-th substep is to construct a random unit vector $\bar v_{k,\ell}$ such that}
\begin{align}
&\begin{aligned}&{\bf 1.}\;\bar v_{k,\ell}\mbox{ is supported on }\prod_{(p,q)\precsim (k,\ell)}\R^{J^p_q},\mbox{ where
the notation}\\
&\mbox{$(p,q)\precsim (k,\ell)$ means ``$p<k$ or $p=k,\,q\le \ell$''};
\end{aligned}\label{nkl prop 1}\\
&\begin{aligned}&\vphantom{{A^{\int\limits^1}}}{\bf 2.}\;\mbox{$\bar v_{k,\ell}$ is measurable with respect to the $\sigma$-algebra generated by}\\
&\mbox{$\Proj^p_q (\B_n(t))$ for all $(p,q)\precsim (k,\ell)$ and $t\in
\bigcup_{i=0}^N\bigl(\{a_{i+1}\}\cup I^i_k\bigr)$};
\end{aligned}\label{nkl prop 2}\\
&\vphantom{A^{\int\limits^1}}\begin{aligned}{\bf 3.}\;\mbox{$\|\BStat(k,\ell+1)\|$ is typically smaller than $\|\BStat(k,\ell)\|$.}
\end{aligned}\nonumber
\end{align}
The third property will be made more precise later. For now, we note that
the typical value of $\|\BStat(k,\ell)\|$ will decrease with $\ell$ 
in such a way that, after the $M'$-th substep, the vector $\BStat(k,M'+1)$ will be zero with probability close to one.

The vector $\bar v_{k,\ell}$ will be defined as
\begin{equation}\label{def of nkl via delta}
\bar v_{k,\ell}=\frac{\bar{v}_{k,\ell-1}+\alpha_{k,\ell}\bar\Delta_{k,\ell}}{\sqrt{1+{\alpha_{k,\ell}}^2}},
\end{equation}
where $\bar\Delta_{k,\ell}$ is a random unit vector (perturbation)
and $\alpha_{k,\ell}:=16^{-k-\ell}$.\\
{\bf The vector $\bar\Delta_{k,\ell}$ will satisfy the following properties:}
\begin{align}
&\begin{aligned}{\bf 1.}\;\bar\Delta_{k,\ell}\mbox{ is supported on }\R^{J^k_\ell};\end{aligned}\label{deltakl prop 1}\\
&\begin{aligned}&\vphantom{{A^{\int\limits^1}}}{\bf 2.}\;
\mbox{$\bar\Delta_{k,\ell}$ is measurable with respect to the $\sigma$-algebra generated by}\\
&\mbox{$\Proj^p_q (\B_n(t))$
for all admissible $(p,q)\precsim (k,\ell)$, $t\in\bigcup_{i=0}^N\bigl(\{a_{i+1}\}\cup I^i_k\bigr)$};
\end{aligned}\label{deltakl prop 2}\\
&\begin{aligned}&\vphantom{{A^{\int\limits^1}}}{\bf 3.}\;
\mbox{For any subset $I\subset\{0,1,\dots,N\}$ such that $\P\{\BadBlocks(k,\ell)=I\}>0$,}\\
&\mbox{$\bar\Delta_{k,\ell}$ is {\it conditionally} independent from the collection of vectors}\\
&\hspace{1cm}\bigl\{\Proj^k_\ell(\B_n(t)-\B_n(a_i)),\;\;t\in I^i_k\cup\{a_{i+1}\},\;\;i\notin I\bigr\}\\
&\mbox{given the event $\{\BadBlocks(k,\ell)=I\}$.}
\end{aligned}\label{deltakl prop 3}\\
&\begin{aligned}
&\vphantom{{A^{\int\limits^1}}}{\bf 4.}\;
\mbox{The event}\\
&\hspace{2cm}\Event_{k,\ell}:=\bigl\{\BStat_i(k,\ell+1)=0\mbox{ for all }i\in\BadBlocks(k,\ell)\bigr\}\\
&\mbox{has probability close to one.}
\end{aligned}\nonumber
\end{align}
Again, we will make the last property more precise later.

Let us sum up the construction procedure. We sequentially produce random unit vectors
$\bar v_0=\bar v_{1,0}$, $\bar v_{1,1}$, $\bar v_{1,2},\dots$, $\bar v_{1,M'}=\bar v_1=\bar v_{2,0}$,
$\bar v_{2,1}$, $\bar v_{2,2},\dots$, $\bar v_{2,M'}=\bar v_2=\bar v_{3,0},\dots$, $\dots$, $\bar v_{M,M'}=\bar v_M$
(in the given order). Each next vector is a random perturbation of the previous one.
In a certain sense (quantified with help of order statistics $\BStat(k,\ell)$),
each newly produced vector is a refinement of the previous one in such a way
that $\bar v_M=\bar v$ will possess the required characteristics.

In the next two lemmas, we establish certain important properties of the block statistics.

\begin{lemma}[Initial substep for block statistics]\label{BStat initial substep lemma}
Fix any $1\le k\le M$ and assume that a random unit vector $\bar v_{k,0}:=\bar v_{k-1}$ satisfying
properties \eqref{nk prop 1} and \eqref{nk prop 2} has been constructed. Then
\begin{align*}
\P&\Bigl\{|\BadBlocks(k,1)|\le N\exp(-{C_h}^2 2^{k/2}/16)\mbox{ and }
\|\BStat(k,1)\|\le\frac{8\sqrt{N}}{\exp({C_h}^2 2^{k/2}/32)}\Bigr\}\\
&\ge \P(\Event_{k-1})-2\exp(-2\sqrt{N}).
\end{align*}
\end{lemma}
\begin{proof}
Let $i>0$ so that $I^i_k\neq \emptyset$.
For each $t\in I^i_k\setminus I^i_{k-1}$, let $t_L$ be the maximal number in $\{a_i\}\cup I^i_{k-1}$ strictly less
than $t$ (``left neighbour'') and, similarly, $t_R$ be the minimal number in $I^i_{k-1}\cup\{a_{i+1}\}$
strictly greater than $t$ (``right neighbour''). For every such $t$, let
$$w_t:=\frac{t_R-t}{t_R-t_L}\B_n(t_L)+\frac{t-t_L}{t_R-t_L}\B_n(t_R);\;\;u_t:=\B_n(t)-w_t.$$
It is not difficult to see that
\begin{align*}
\bigl\langle &\bar v_{k,0},\frac{\B_n(a_i)-w_t}{\sqrt{a_i}}\bigr\rangle\\
&\le \max\Bigl(\bigl\langle \bar v_{k,0},\frac{\B_n(a_i)-\B_n(t_L)}{\sqrt{a_i}}\bigr\rangle,
\bigl\langle \bar v_{k,0},\frac{\B_n(a_i)-\B_n(t_R)}{\sqrt{a_i}}\bigr\rangle\Bigr)\\
&\le \max\Bigl(0,
\max\limits_{\tau\in I^i_{k-1}}\bigl\langle \bar v_{k,0},\frac{\B_n(a_i)-\B_n(\tau)}{\sqrt{a_i}}\bigr\rangle,\\
&\hspace{1.7cm}\bigl\langle 2\bar v_{k,0},\frac{\B_n(a_i)-\B_n(a_{i+1})}{\sqrt{a_{i+1}}}\bigr\rangle\Bigr).
\end{align*}
Hence, the $i$-th block statistic (for $i=0,1,\dots,N$) can be (deterministically) bounded as
\begin{align*}
\BStat_i(k,1)
&\le\max\Bigl(0,
\max\limits_{t\in I^i_{k-1}}\bigl\langle \bar v_{k,0},
\frac{\B_n(a_i)-\B_n(t)}{\sqrt{a_i}}\bigr\rangle-h(k,1),\Bigr.\\
&\hspace{0.5cm}\max\limits_{t\in I^i_{k}\setminus I^i_{k-1}}\bigl\langle \bar v_{k,0},
\frac{\B_n(a_i)-w_t}{\sqrt{a_i}}\bigr\rangle-h(k,1)
+\max\limits_{t\in I^i_k\setminus I^i_{k-1}}\bigl\langle\bar v_{k,0},\frac{-u_t}{\sqrt{a_i}}\bigr\rangle,\Bigr.\\
\Bigl.&\hspace{0.5cm}\bigl\langle \bar v_{k,0},\frac{\B_n(a_i)-\B_n(a_{i+1})}{\sqrt{a_{i+1}}}\bigr\rangle+f(k,1)\Bigr)\\
&\le\max\Bigl(0,
\max\limits_{t\in I^i_{k-1}}\bigl\langle \bar v_{k,0},
\frac{\B_n(a_i)-\B_n(t)}{\sqrt{a_i}}\bigr\rangle-h(k,0),\Bigr.\\
&\hspace{0.5cm}\Bigl.\bigl\langle 2\bar v_{k,0},\frac{\B_n(a_i)-\B_n(a_{i+1})}{\sqrt{a_{i+1}}}\bigr\rangle+2f(k,0)\Bigr)\\
&\hspace{0.5cm}+\max\Bigl(0,\max\limits_{t\in I^i_k\setminus I^i_{k-1}}\bigl\langle\bar v_{k,0},\frac{-u_t}{\sqrt{a_i}}\bigr\rangle+h(k,0)-h(k,1)\Bigr).
\end{align*}
Let us denote the first summand in the last estimate by $\xi_i$, so that
$$\BStat_i(k,1)\le \xi_i+\max\Bigl(0,\max\limits_{t\in I^i_k\setminus I^i_{k-1}}\bigl\langle\bar v_{k,0},
\frac{-u_t}{\sqrt{a_i}}\bigr\rangle+h(k,0)-h(k,1)\Bigr).$$
Note that
\begin{equation}\label{BStat lem aux 1}
\Event_{k-1}=\bigl\{\xi_i=0\mbox{ for all }i=0,1,\dots,N\bigr\}.
\end{equation}
Further, the property \eqref{nk prop 2} of the vector $\bar v_{k,0}=\bar v_{k-1}$, together with Lemma~\ref{brownian-segment-split}
and the independence of the Brownian motion on disjoint intervals, imply that
the Gaussian variables $\bigl\langle\bar v_{k,0},\frac{-u_t}{\sqrt{a_i}}\bigr\rangle$
are jointly independent for $t\in I^i_k\setminus I^i_{k-1}$, $i=1,2,\dots,N$, and the variance of each one
can be estimated from above by $2^{1-k}$. Thus, the vector $\BStat(k,1)$ can be majorized coordinate-wise
by the vector
$$\bigl(\xi_i+\max\limits_{t\in I^i_k\setminus I^i_{k-1}}(0,2^{(1-k)/2}g_{t}+h(k,0)-h(k,1))\bigr)_{i=0}^N,$$
where $g_t$ ($t\in I^i_k\setminus I^i_{k-1}$, $i=0,1,\dots,N$) are i.i.d.\ standard Gaussians
(in fact, appropriate scalar multiples of $\bigl\langle\bar v_{k,0},\frac{-u_t}{\sqrt{a_i}}\bigr\rangle$).
Denoting by $g$ the standard Gaussian variable, we get from the definition of $h$:
\begin{align*}
\P\bigl\{\max\limits_{t\in I^i_k\setminus I^i_{k-1}}(0,2^{(1-k)/2}g_{t}+h(k,0)-h(k,1))>0\bigr\}
&\le 2^k\P\{g>C_h2^{k/4}/2\}\\
&\le 2^k\exp(-{C_h}^2 2^{k/2}/8)\\
&\le\frac{1}{2}\exp(-{C_h}^2 2^{k/2}/16).
\end{align*}
(In the last two inequalities, we assumed that $C_h$ is sufficiently large).
Applying Hoeffding's inequality to corresponding indicators, we infer
$$|\BadBlocks(k,1)|\le |\{i:\,\xi_i\neq 0\}|+N\exp(-{C_h}^2 2^{k/2}/16)$$
with probability at least $1-\exp(-2\sqrt{N})$ (we note that, in view of the inequality $k\le M$, we have
$\frac{1}{2}\exp(-{C_h}^2 2^{k/2}/16)\ge N^{-1/4}$).
Next, it is not hard to see that the Euclidean norm of $\BStat(k,1)$ is majorized (deterministically) by the sum
$$\bigl\|(\xi_i)_{i=0}^N\bigr\|+2^{(1-k)/2}\bigl\|\bigl(\max(0,g_t-C_h2^{k/4}/2)\bigr)_{t}\bigr\|,$$
with the second vector having $\sum_{i=0}^N|I^i_k\setminus I^i_{k-1}|\le 2^kN$ coordinates.
Applying Lemma~\ref{norm of noncentered gaussian} to the second vector
(note that for sufficiently large $n$ we have $C_h2^{k/4}/2\le\sqrt{\ln N}$), we get
$$\|\BStat(k,1)\|\le \bigl\|(\xi_i)_{i=0}^N\bigr\|+\frac{8 \sqrt{N}}{\exp({C_h}^2 2^{k/2}/32)}$$
with probability at least $1-\exp(-2\sqrt{N})$.
Combining the estimates with \eqref{BStat lem aux 1}, we obtain the result.
\end{proof}

\begin{lemma}[Subsequent substeps for block statistics]\label{BStat subsequent substeps lemma}
Fix any $1\le k\le M$ and $1<\ell\le M'+1$ and assume that the random unit vectors
$\bar v_{k,\ell-2}$ and $\bar \Delta_{k,\ell-1}$ satisfying properties \eqref{nkl prop 1}---\eqref{nkl prop 2}
and \eqref{deltakl prop 1}---\eqref{deltakl prop 2}---\eqref{deltakl prop 3}, respectively,
are constructed, and $\bar v_{k,\ell-1}$ is defined according to formula \eqref{def of nkl via delta}. Then
\begin{align*}
\P&\Bigl\{|\BadBlocks(k,\ell)|\le N\exp(-{C_h}^2 2^{(k+\ell)/2})\mbox{ and }
\|\BStat(k,\ell)\|\le\frac{\sqrt{N}}{\exp({C_h}^2 2^{(k+\ell)/2})}\Bigr\}\\
&\ge \P(\Event_{k,\ell-1})-2\exp(-2\sqrt{N}).
\end{align*}
Moreover,
$$\P\bigl\{\BadBlocks(k,\ell)\neq\emptyset\bigr\}\le N\exp(-{C_h}^2/\alpha_{k,\ell-1})+1-\P(\Event_{k,\ell-1}).$$
\end{lemma}
\begin{proof}
To shorten the notation, we will use $\alpha$ in place of $\alpha_{k,\ell-1}$ within the proof.
Using the definition of $\bar v_{k,\ell-1}$ in terms of $\bar v_{k,\ell-2}$ and $\bar \Delta_{k,\ell-1}$, we get
for every $i=0,1,\dots,N$
\begin{align*}
\BStat_i&(k,\ell)=\max\Bigl(0,
\max\limits_{t\in I^i_k}\bigl\langle \frac{\bar v_{k,\ell-2}+\alpha\bar\Delta_{k,\ell-1}}{\sqrt{1+\alpha^2}},
\frac{\B_n(a_i)-\B_n(t)}{\sqrt{a_i}}\bigr\rangle-h(k,\ell),\Bigr.\\
&\hspace{2cm}\Bigl.\bigl\langle \frac{\bar v_{k,\ell-2}
+\alpha\bar\Delta_{k,\ell-1}}{\sqrt{1+\alpha^2}},\frac{\B_n(a_i)-\B_n(a_{i+1})}{\sqrt{a_{i+1}}}\bigr\rangle+f(k,\ell)\Bigr)\\
&\le\frac{\BStat_i(k,\ell-1)}{\sqrt{1+\alpha^2}}\\
&\hspace{0.3cm}+\max\Bigl(0,\max\limits_{t\in I^i_k}\bigl\langle \alpha\bar\Delta_{k,\ell-1},
\frac{\B_n(a_i)-\B_n(t)}{\sqrt{a_i}}\bigr\rangle+h(k,\ell-1)-h(k,\ell),\Bigr.\\
&\hspace{1cm}\Bigl.\bigl\langle \alpha\bar\Delta_{k,\ell-1},\frac{\B_n(a_i)-\B_n(a_{i+1})}{\sqrt{a_{i+1}}}\bigr\rangle+\sqrt{1+\alpha^2}f(k,\ell)-
f(k,\ell-1)\Bigr).
\end{align*}
Let us denote the second summand by $\eta_i$ so that
$$\BStat_i(k,\ell)\le\frac{\BStat_i(k,\ell-1)}{\sqrt{1+\alpha^2}}+\eta_i.$$
Fix for a moment any subset $I$ of $\{0,1,\dots,N\}$ such that $\P\{\BadBlocks(k,\ell-1)=I\}>0$.
A crucial observation is that, conditioned on the event $\BadBlocks(k,\ell-1)=I$, the variables
$\eta_i$, $i\notin I$, are jointly independent. This follows from properties \eqref{deltakl prop 1}, \eqref{deltakl prop 3}
of $\bar\Delta_{k,\ell-1}$ and from independence of the Brownian motion on disjoint intervals.
Next, the same properties tell us that, conditioned on $\BadBlocks(k,\ell-1)=I$,
each variable $\langle\bar\Delta_{k,\ell-1},\frac{\B_n(a_i)-\B_n(t)}{\sqrt{a_i}}\rangle$, $t\in I^i_k$,
and $\langle\bar\Delta_{k,\ell-1},\frac{\B_n(a_i)-\B_n(a_{i+1})}{\sqrt{a_{i+1}}}\rangle$
have Gaussian distributions with variances at most $1$.
Further, note that, by the choice of $\alpha$ and the functions $f$ and $h$,
we have
$$\sqrt{1+\alpha^2}f(k,\ell)-f(k,\ell-1)\le h(k,\ell-1)-h(k,\ell)=-C_h 2^{(-k-\ell)/4}.$$
Thus, denoting by $g$ the standard Gaussian variable, we get
\begin{align}
\P\{\eta_i>0\,|\,\BadBlocks(k,\ell-1)=I\}&\le 2^k\P\{g>\alpha^{-1}C_h 2^{(-k-\ell)/4}\}\nonumber\\
&\le \frac{1}{2}\exp(-{C_h}^2\alpha^{-1}),\;\;i\in\{0,1,\dots,N\}\setminus I.\label{aux 167}
\end{align}
Hence, by Hoeffding's inequality (note that $\exp(-{C_h}^2 2^{(k+\ell)/2})> 2N^{-1/4}$):
$$\P\bigl\{|\{i\notin I:\,\eta_i>0\}|\ge N\exp(-{C_h}^2 2^{(k+\ell)/2})\,|\,\BadBlocks(k,\ell-1)=I\bigr\}\le\exp(-2\sqrt{N}).$$
Next, it is not difficult to see that for any $\tau>0$ and $i\notin I$
\begin{align*}
\P&\{\eta_i^2\ge\tau\,|\,\BadBlocks(k,\ell-1)=I\}\\
&\le 2^k\P\{\max(0,\alpha g-C_h 2^{(-k-\ell)/4})^2\ge\tau\}\\
&\le 1-\exp\bigl(-2^{k+1}\P\{\max(0,\alpha g-C_h 2^{(-k-\ell)/4})^2\ge\tau\}\bigr)\\
&\le 1-\P\bigl\{\max(0,\alpha g-C_h 2^{(-k-\ell)/4})^2<\tau\bigr\}^{2^{k+1}}\\
&\le \P\Bigl\{\sum\limits_{j=1}^{2^{k+1}}\max(0,\alpha g_j-C_h 2^{(-k-\ell)/4})^2\ge\tau\Bigr\}\\
&\le\P\Bigl\{\sum\limits_{j=1}^{2^{k+1}}\max(0,\alpha g_j-4\alpha C_h 2^{(k+\ell)/4})^2\ge\tau\Bigr\},
\end{align*}
where $ g_j$ ($j=1,2,\dots,2^{k+1}$) are i.i.d.\ copies of $ g$. Hence, the conditional cdf of $\|(\eta_i)_{i\notin I}\|$
given $\BadBlocks(k,\ell-1)=I$ majorizes the cdf of
$$\alpha\bigl\|\bigl(\max(0, g_j-4C_h 2^{(k+\ell)/4})\bigr)_{j=1}^{2^{k+1}N}\bigr\|=:\alpha Z$$
for i.i.d.\ standard Gaussians $ g_j$, $j=1,2,\dots,2^{k+1}N$. Applying Lemma~\ref{norm of noncentered gaussian}
(note that $4C_h 2^{(k+\ell)/4}\le\sqrt{\ln N}$),
we obtain
\begin{align*}
\P&\Bigl\{\|(\eta_i)_{i\notin I}\|>\frac{\sqrt{N}}{\exp({C_h}^2 2^{(k+\ell)/2})}\,\bigl|\bigr.\,\BadBlocks(k,\ell-1)=I\Bigr\}\\
&\le\P\Bigl\{Z>
\frac{\alpha^{-1}\sqrt{N}}{\exp({C_h}^2 2^{(k+\ell)/2})}\,\bigl|\bigr.\,\BadBlocks(k,\ell-1)=I\Bigr\}\\
&\le\P\Bigl\{Z>
\frac{4\sqrt{2^{k+1}N}}{\exp(2{C_h}^2 2^{(k+\ell)/2})}\,\bigl|\bigr.\,\BadBlocks(k,\ell-1)=I\Bigr\}\\
&\le \exp(-2\sqrt{N}).
\end{align*}
Clearly $\BStat_i(k,\ell-1)=0$ for all $i\notin I$ given $\BadBlocks(k,\ell-1)=I$.
Hence, the above estimates give
\begin{align*}
\P\Bigl\{&|\BadBlocks(k,\ell)|\ge N\exp(-{C_h}^2 2^{(k+\ell)/2})\Bigr.\\
\Bigl.&\mbox{or }\|\BStat(k,\ell)\|>\frac{\sqrt{N}}{\exp({C_h}^2 2^{(k+\ell)/2})}\,\bigl|\bigr.\,\BadBlocks(k,\ell-1)=I\Bigr\}\\
&\le \P\bigl\{\BStat_i(k,\ell)>0\mbox{ for some }i\in I\,|\,\BadBlocks(k,\ell-1)=I\bigr\}+2\exp(-2\sqrt{N}).
\end{align*}
Summing over all admissible subsets $I$, we get
\begin{align*}
\P&\Bigl\{|\BadBlocks(k,\ell)|\ge N\exp(-{C_h}^2 2^{(k+\ell)/2})
\mbox{ or }\|\BStat(k,\ell)\|>\frac{\sqrt{N}}{\exp({C_h}^2 2^{(k+\ell)/2})}\Bigr\}\\
&\le 2\exp(-2\sqrt{N})\\
&\hspace{0.5cm}+\sum\limits_I \P\bigl\{\BStat_i(k,\ell)>0\mbox{ for some }i\in I\,|\,\BadBlocks(k,\ell-1)=I\bigr\}\P\{\BadBlocks(k,\ell-1)=I\}\\
&=2\exp(-2\sqrt{N})+\P\bigl\{\BStat_i(k,\ell)>0\mbox{ for some }i\in \BadBlocks(k,\ell-1)\bigr\}\\
&=2\exp(-2\sqrt{N})+1-\P(\Event_{k,\ell-1}).
\end{align*}
By analogous argument, as a corollary of \eqref{aux 167},
$$\P\bigl\{\BadBlocks(k,\ell)\neq\emptyset\bigr\}\le N\exp(-{C_h}^2\alpha^{-1})+1-\P(\Event_{k,\ell-1}).$$
\end{proof}

The next lemma, which is the heart of the proof,
provides a construction procedure for the perturbation $\bar\Delta_{k,\ell}$. 
Given vector $\bar v_{k,\ell-1}$, we examine its block statistics $\BStat(k,\ell)$,
and define the perturbation in such a way that its inner product with
increments of the Brownian motion is large on bad blocks $\BadBlocks(k,\ell)$
(in fact, it will be proportional to the values of corresponding $\BStat_i(k,\ell)$),
and random on other blocks.
This is achieved using Lemma~\ref{normal-vector-lem}.

\begin{lemma}[Construction of $\bar\Delta_{k,\ell}$]\label{construction of deltakl}
Let $1\le k\le M$ and $1\le \ell\le M'$ and assume that the random unit vector
$\bar{v}_{k,\ell-1}$ satisfying properties \eqref{nkl prop 1} and \eqref{nkl prop 2}
has been constructed.
Then one can construct a random unit vector $\bar \Delta_{k,\ell}$ satisfying properties
\eqref{deltakl prop 1}---\eqref{deltakl prop 2}---\eqref{deltakl prop 3} and such that
\begin{equation*}
\begin{split}
\P(\Event_{k,\ell})\ge \P(\Event_{k,\ell-1})-3\exp(-\sqrt{N})\;\;\;\;\;\;&\mbox{if $\ell>1$, or}\\
\P(\Event_{k,\ell})\ge \P(\Event_{k-1})-3\exp(-\sqrt{N})\;\;\;\;\;\;&\mbox{if $\ell=1$.}
\end{split}
\end{equation*}
\end{lemma}
\begin{proof}
Fix for a moment any subset $I\subset\{0,1,\dots,N\}$ such that the event
$$\Gamma_I=\{\BadBlocks(k,\ell)=I\}$$
has a non-zero probability.
If $|I|>N\exp(-{C_h}^2 2^{(k+\ell)/2}/32)$ then define a random vector $\bar \Delta_{k,\ell}^I$
on $\Gamma_I$ by setting $\bar \Delta_{k,\ell}^I:=u$ for an arbitrary fixed unit vector $u\in\R^{J^k_\ell}$.
Otherwise, if $|I|\le N\exp(-{C_h}^2 2^{(k+\ell)/2}/32)$, we proceed as follows:

Define a set of double indices
$$
T_I:= \big\{(i,p):\, i\in I\setminus \{0\},\, p\in\{1,\ldots, 2^{k}-1\}\big\} \cup \bigcup_{i\in I}\{(i,0)\}.  
$$
For each $(i,p)\in T_I$, define an increment $X_{i,p}$ on the probability space $(\Gamma_I,\P(\cdot|\Gamma_I))$ by
$$
X_{i,p}:= \frac{\Proj_\ell^k\bigl(\B_n(t_{i,p+1})- \B_n(t_{i,p})\bigr)}{\sqrt{t_{i,p+1}-t_{i,p}}},
$$
where $t_{i,p}= 2^{i-1+p2^{-k}}$ for $p=0,1,\dots,2^k$ and $i\in I\setminus\{0\}$; additionally, if $0\in I$, then
$t_{0,1}=1$ and $t_{0,0}=0$.

Note that $\BStat(k,\ell)$ is measurable with respect to the $\sigma$-algebra generated by processes 
$\Proj_s^q \B_n(t)$, $(q,s)\precsim (k,\ell -1)$, where the notation ``$\precsim$'' is taken from 
\eqref{nkl prop 1}; see formula \eqref{def-statistics}. 
It implies that $\Proj^k_\ell(\B_n(t))$ (on $\Omega$) is independent from the event $\Gamma_I$; moreover,
considered on the space $(\Gamma_I,\P(\cdot|\Gamma_I))$, the set
$\{X_{i,p},\;(i,p)\in T_I\}$ is a collection of standard Gaussian vectors,
such that all $X_{i,p}$ and the vector $\BStat(k,\ell)$ are 
{\it jointly independent}.
Let us define a random vector $\tilde b^I\in\R^{T_I}$ on $(\Gamma_I,\P(\cdot|\Gamma_I))$ by
$$\tilde b^I_{i,p}=\begin{cases}2^{-k/2}\BStat_i(k,\ell)/\|\BStat(k,\ell)\|,&\mbox{if }\BStat(k,\ell)\neq {\bf 0};\\
0,&\mbox{otherwise.}\end{cases}$$
It is easy to see that $\|\tilde b^I\|\le 1$ (deterministically) and that
$$|T_I|\le 2^k|I|\le 2^k N\exp(-{C_h}^2 2^{(k+\ell)/2}/32)\le \frac{1}{2}|J^k_\ell|.$$
(In the last estimate, we used the assumption that $C_h$ is a large constant).
Hence, in view of Lemma~\ref{normal-vector-lem}, there exists a random unit vector
$\bar\Delta_{k,\ell}^I$ on the space $(\Gamma_I,\P(\cdot|\Gamma_I))$ with values in $\R^{J_\ell^k}$,
which is a Borel function of $X_{i,p}$ and $\tilde b^I$, and such that
\begin{align*}
\P\Bigl\{\langle \bar\Delta_{k,\ell}^I,X_{i,p}\rangle
\ge c_{\ref{normal-vector-lem}}\sqrt{|J_\ell^k|}\;\tilde b^I_{i,p}\mbox{ for all }(i,p)\in T_I\,|\,\Gamma_I\Bigr\}
&\ge 1-\exp(-c_{\ref{normal-vector-lem}}|J_\ell^k|)\\
&\ge 1-\exp(-\sqrt{N}).
\end{align*}
It will be convenient for us to denote by $\tilde \Gamma_I$ the event
$$\Bigl\{\langle \bar\Delta_{k,\ell}^I,X_{i,p}\rangle
\ge c_{\ref{normal-vector-lem}}\sqrt{|J_\ell^k|}\;\tilde b^I_{i,p}\mbox{ for all }(i,p)\in T_I\Bigr\}\subset\Gamma_I.$$

\medskip

By glueing together $\bar\Delta_{k,\ell}^I$ for all $I$,
we obtain a random vector $\bar\Delta_{k,\ell}$ defined on the entire probability space $\Omega$.

Clearly, $\bar\Delta_{k,\ell}$ satisfies properties \eqref{deltakl prop 1} and \eqref{deltakl prop 2}.
Next, on each $\Gamma_I$ with $\P(\Gamma_I)>0$ the vector $\bar\Delta_{k,\ell}$ was defined
as a Borel function of $\BStat(k,\ell)$ and $\Proj^k_\ell(\B(t)-\B(\tau))$, $t,\tau\in I^i_k\cup\{a_i,a_{i+1}\}$, $i\in I$, so,
in view of independence of the Brownian motion on disjoint intervals, $\bar\Delta_{k,\ell}$ satisfies \eqref{deltakl prop 3}.

Finally, we shall estimate the probability of $\Event_{k,\ell}$.  
Define
\begin{align*}
\Event=\Bigl\{&|\BadBlocks(k,\ell)|\le N\exp(-{C_h}^2 2^{(k+\ell)/2}/32)\mbox{ and }\\
&\|\BStat(k,\ell)\|\le\frac{\sqrt{N}}{\exp({C_h}^2 2^{(k+\ell)/2}/64)}\Bigr\}.
\end{align*}
Note that, according to Lemmas~\ref{BStat initial substep lemma} and~\ref{BStat subsequent substeps lemma},
the probability of $\Event$ can be estimated from below by $\P(\Event_{k,\ell-1})-2\exp(-2\sqrt{N})$
for $\ell>1$ and $\P(\Event_{k-1})-2\exp(-2\sqrt{N})$ for $\ell=1$.

Take any subset $I\subset\{0,1,\dots,N\}$ with $|I|\le N\exp(-{C_h}^2 2^{(k+\ell)/2}/32)$ and
such that $\tilde\Gamma_I\cap \Event\neq\emptyset$, and let $\omega\in\tilde \Gamma_I\cap\Event$.
If $\BadBlocks(k,\ell)=\emptyset$ at point $\omega$ then, obviously, $\omega\in\Event_{k,\ell}$.
Otherwise, we have
\begin{align*}
\bigl\langle &\bar\Delta_{k,\ell}(\omega),\frac{\B_n(t_{i,p+1})(\omega)- \B_n(t_{i,p})(\omega)}{\sqrt{t_{i,p+1}-t_{i,p}}}
\bigr\rangle\\
&\ge \frac{c_{\ref{normal-vector-lem}}2^{-k/2}\sqrt{|J_\ell^k|}\;\BStat_i(k,\ell)(\omega)}{\|\BStat(k,\ell)(\omega)\|}
\mbox{ for all }(i,p)\in T_I,
\end{align*}
whence, using the estimate $t_{i,p+1}-t_{i,p}\ge \frac{2^{i-k}}{4}$ ($(i,p)\in T_I$), we obtain for
any $i\in I$ and $t\in I^i_k\cup\{a_{i+1}\}$:
\begin{align*}
\bigl\langle &\bar\Delta_{k,\ell}(\omega),\B_n(t)(\omega)-\B_n(a_i)(\omega)\bigr\rangle\\
&=\sum_{p:\,t_{i,p}<t}  \bigl\langle \bar\Delta_{k,\ell}(\omega),\B_n(t_{i,p+1})(\omega)-\B_n(t_{i,p})(\omega)\bigr\rangle\\
&\ge \frac{c_{\ref{normal-vector-lem}}2^{-k-1}\sqrt{a_{i+1}|J_\ell^k|}\;\BStat_i(k,\ell)(\omega)}{\|\BStat(k,\ell)(\omega)\|}.
\end{align*}
Further,
$$\frac{c_{\ref{normal-vector-lem}}2^{-k-1}\sqrt{|J_\ell^k|}}{\|\BStat(k,\ell)(\omega)\|}
\ge \frac{c_{\ref{normal-vector-lem}}2^{-k-1}\sqrt{c' n 2^{(-k-\ell)/8}}\exp({C_h}^2 2^{(k+\ell)/2}/64)}{\sqrt{N}}\ge \frac{1}{\alpha_{k,\ell}}.
$$
Using the definition of $\bar v_{k,\ell}$ in terms of $\bar v_{k,\ell-1}$ and $\bar\Delta_{k,\ell}$ and the above estimates,
we get
\begin{align*}
\bigl\langle &\bar v_{k,\ell}(\omega),\frac{\B_n(t)(\omega)-\B_n(a_i)(\omega)}{\sqrt{a_i}}\bigr\rangle\\
&\ge
\frac{\alpha_{k,\ell}}{\sqrt{1+{\alpha_{k,\ell}}^2}}\bigl\langle \bar \Delta_{k,\ell}(\omega),
\frac{\B_n(t)(\omega)-\B_n(a_i)(\omega)}{\sqrt{a_i}}\bigr\rangle-
\frac{h(k,\ell)+\BStat_i(k,\ell)(\omega)}{\sqrt{1+{\alpha_{k,\ell}}^2}}\\
&\ge\frac{-h(k,\ell)}{\sqrt{1+{\alpha_{k,\ell}}^2}}\\
&\ge -h(k,\ell+1),\;\;t\in I^i_k,\;\;i\in I,
\end{align*}
and, similarly,
$$\bigl\langle \bar v_{k,\ell}(\omega),\frac{\B_n(a_{i+1})(\omega)-\B_n(a_i)(\omega)}{\sqrt{a_{i+1}}}\bigr\rangle
\ge \frac{f(k,\ell)}{{\sqrt{1+{\alpha_{k,\ell}}^2}}}\ge f(k,\ell+1),\;\;i\in I.$$
Thus, by the definition of the event $\Event_{k,\ell}$, we get $\omega\in\Event_{k,\ell}$.

The above argument shows that
$$\P(\Event_{k,\ell})\ge\sum\limits_I\P(\tilde \Gamma_I\cap\Event),$$
where the sum is taken over all $I$ with $|I|\le N\exp(-{C_h}^2 2^{(k+\ell)/2}/32)$.
Finally,
$$\sum\limits_I\P(\tilde \Gamma_I\cap\Event)\ge\sum\limits_{I}\P(\Gamma_I\cap\Event)
-\sum\limits_I\P(\Gamma_I\setminus\tilde\Gamma_I)\ge\P(\Event)-\exp(-\sqrt{N}),$$
and we get the result.
\end{proof}

\begin{proof}[Proof of Lemma~\ref{k step lemma}]
As before, we set $\bar v_{k,0}:=\bar v_{k-1}$. Consecutively applying Lemma~\ref{construction of deltakl}
and formula \eqref{def of nkl via delta} $M'$ times, we obtain a random unit vector
$\bar v_{k,M'}$ satisfying \eqref{nkl prop 1} and \eqref{nkl prop 2}.
Moreover, the same lemma provides the estimate
$$\P(\Event_{k,\M'})\ge\P(\Event_{k-1})-3M'\exp(-\sqrt{N}).$$
Then, in view of Lemma~\ref{BStat subsequent substeps lemma} and the definition of $M'$, we have
$$\P\bigl\{\BadBlocks(k,M'+1)\neq\emptyset\bigr\}\le N\exp(-{C_h}^2/\alpha_{k,M'})+1-\P(\Event_{k,M'})\le
\frac{1}{n^2}+1-\P(\Event_{k-1}).$$
Combining the above estimate with the definition of $\Event_k$, we get for $\bar v_k:=\bar v_{k,M'}$ that
$$\P(\Event_k)\ge \P(\Event_{k-1})-\frac{1}{n^2}.$$
\end{proof}

{\bf Acknowledgements.} The first named author would like to thank Ronen Eldan
for introducing him to the question. Both authors are grateful to Nicole Tomczak-Jaegermann, Ronen Eldan
and Olivier Gu\' edon for discussions and valuable suggestions. 
Finally, the authors would like to thank the referee for valuable remarks and suggestions which helped improve the manuscript.


\end{document}